\documentclass[11pt]{article}

\usepackage[T1]{fontenc}
\usepackage[utf8]{inputenc}
\usepackage{amsmath,amssymb,amsthm}
\usepackage{enumitem}
\usepackage{geometry}
\usepackage{graphicx}
\usepackage{caption}
\usepackage{hyperref}
\usepackage{mathtools}
\usepackage{mathrsfs}

\numberwithin{equation}{section}

\theoremstyle{plain}
\newtheorem{theorem}{Theorem}[section]
\newtheorem{lemma}[theorem]{Lemma}

\newtheorem{corollary}[theorem]{Corollary}

\theoremstyle{definition}
\newtheorem{definition}[theorem]{Definition}

\newtheorem{remark}[theorem]{Remark}
\title{Nonlinear Transmission Eigenvalue Problems with Nonhomogeneous Operators of Different $p$-Growth}

\author{
Lumini\c{t}a Barbu\thanks{Faculty of Mathematics and Computer Science,
Ovidius University of Constan\c{t}a, 124 Mamaia Blvd., 900527 Constan\c{t}a, Romania.
E-mail: \texttt{lbarbu@univ-ovidius.ro}; \texttt{raluca.turtoi@365.univ-ovidius.ro}}
\and
Raluca-Gabriela Turtoi\footnotemark[1]
}

\date{\today}

\begin{document}

\maketitle

\begin{abstract}
Let $\Omega \subset \mathbb{R}^N$, $N \ge 2$, be a bounded domain with Lipschitz boundary,
divided by a Lipschitz hypersurface $\Sigma$ into two open, disjoint Lipschitz subdomains
$\Omega_1$ and $\Omega_2$.
We study a nonlinear transmission eigenvalue problem driven by nonhomogeneous operators with $p_i$- growth in each subdomain $\Omega_i$, $i=1,2$, and subject to continuity and flux transmission conditions across the interface $\Sigma$.
The real parameter $\lambda$ appears both in the equations and in the nonlinear boundary conditions.
Using variational methods, we prove the existence of an unbounded sequence of eigenvalues.
Under additional assumptions, we establish that the set of eigenvalues coincides with the entire interval $(0,\infty)$.
As a particular case, we obtain the corresponding eigenvalue results for the associated single-domain problem.
\end{abstract}



\section{Introduction}
\label{intro}

Let $\Omega \subset \mathbb{R}^N$, $N \ge 2$, be a bounded domain with Lipschitz boundary.
Assume that $\Omega$ is divided by a Lipschitz hypersurface $\Sigma$
into two open, disjoint Lipschitz subdomains $\Omega_1$ and $\Omega_2$,
such that $\overline{\Omega}=\overline{\Omega}_1\cup\overline{\Omega}_2$ and $\overline{\Sigma}=\overline{\Omega}_1\cap\overline{\Omega}_2$. We denote $\Gamma_i=\partial\Omega_i\setminus\overline{\Sigma}$, for $i=1,2$.

We consider two geometric configurations. In the first case, $\Sigma$ is an interior hypersurface, that is,
$\overline{\Sigma}\cap\partial\Omega=\emptyset$. Moreover, $
\Gamma_1\cap\Gamma_2=\emptyset.
$
In the second case, $\overline{\Sigma}\cap\partial\Omega\neq\emptyset$, and, in addition, $\overline{\Gamma}_1\cap\overline{\Gamma}_2=\overline{\Sigma}\cap\partial\Omega$.
We also require that $\mathcal H^{N-1}(\Sigma)>0$ and
$\mathcal H^{N-1}(\Gamma_i)>0$ for $i=1,2$, where $\mathcal H^{N-1}(\cdot)$ denotes the
$(N-1)$- dimensional Hausdorff measure.

We consider the nonlinear transmission eigenvalue problem in $\Omega$ 
\begin{equation}\label{eq:1.0} 
\begin{cases} 
-\operatorname{div}\mathcal{A}(x,\nabla u)+\rho(x)|u|^{\zeta-2}u = \lambda\,\alpha(x)|u|^{q-2}u & \text{in }\Omega, \\ \mathcal{A}(x,\nabla u)\cdot\nu +\gamma(x)|u|^{\eta-2}u = \lambda\,\beta(x)|u|^{q-2}u & \text{on }\partial\Omega , 
\end{cases} 
\end{equation} 
subject to the transmission conditions 
\begin{equation}\label{eq:1.00} 
u_1=u_2, \qquad A_1(x,\nabla u_1)\cdot\nu = A_2(x,\nabla u_2)\cdot\nu \quad \text{on }\Sigma, 
\end{equation} 
where $u_i:=u|_{\Omega_i}$, $A_i$ denotes the operator acting in $\Omega_i$, $i=1,2$, and $\nu$ denotes the unit outward normal on $\partial\Omega$, while on $\Sigma$ it denotes a fixed unit normal oriented from $\Omega_1$ to $\Omega_2$. 

\noindent 
Here, the operator $\mathcal{A}$ is assumed to have $p(x)$–growth, where $p(x):=p_1 \chi_{\Omega_1}(x)+p_2\chi_{\Omega_2}(x)$, that is, for $(x,\xi)\in\Omega\times\mathbb{R}^N$ and $i=1, 2,$ \begin{equation}
\label{eq:Ai} 
\mathcal{A}(x,\xi):= A_1(x,\xi)\chi_{\Omega_1}(x)+A_2(x,\xi)\chi_{\Omega_2}(x),\quad A_i(x,\xi):=G_{iy} \big(x,F(\xi) \big)\nabla_\xi F(\xi), 
\end{equation} 
where $G_i:\Omega_i\times[0,\infty)\to[0,\infty)$ is a Carath\'{e}odory function satisfying a $p_i$-growth condition, while $F:\mathbb{R}^N\to[0,\infty)$ is a given norm. Throughout the paper, $G_{iy}(x,y)$ denotes the partial derivative of $G_i(x,\cdot)$ with respect to $y$, and $\nabla_\xi F(\xi)$ denotes the gradient of $F$ with respect to $\xi$.

Also, all the coefficients may change across the interface $\Sigma$.
More precisely,
\[
\alpha(x):=\alpha_1(x)\chi_{\Omega_1}(x)+\alpha_2(x)\chi_{\Omega_2}(x),
\qquad
\beta(x):=\beta_1(x)\chi_{\Gamma_1}(x)+\beta_2(x)\chi_{\Gamma_2}(x),
\]
and similarly for $\rho$ and $\gamma$.
Moreover, the exponents $q, \zeta, \eta$ are piecewise constant, namely,
\[
\begin{aligned}
q(x) &=
\begin{cases}
q_1, & x\in \Omega_1\cup\Gamma_1,\\
q_2, & x\in \Omega_2\cup\Gamma_2,
\end{cases}
\qquad
\zeta(x) =
\begin{cases}
\zeta_1, & x\in \Omega_1,\\
\zeta_2, & x\in \Omega_2,
\end{cases}
\qquad
\eta(x) =
\begin{cases}
\eta_1, & x\in \Gamma_1,\\
\eta_2, & x\in \Gamma_2.
\end{cases}
\end{aligned}
\]

Additional assumptions on $F$, $G_i$, the exponents and the weight functions are specified in
Subsection~\ref{ss:AFF}.
 
The general framework adopted here allows us to treat several models
arising in the study of nonlinear diffusion in heterogeneous media
in a unified way. We briefly mention some representative special cases.

(i) If $\Omega_2=\emptyset$, the problem reduces to the corresponding
single-domain eigenvalue problem driven by a $p$-growth operator
(see Corollary~\ref{cor:singledomain}). Related eigenvalue results in the
single-domain setting involving $G=(p,r)$-Laplacian operators can be
found in \cite{BBMmjm,BBMauoc}. For $G$ satisfying the assumptions in
Section~\ref{ss:AFF}, see, e.g., \cite{PR,PRS,PW,CPV}.

(ii) If $p_1\neq p_2$ and $G_i$ corresponds to the $p_i$-Laplacian, $i=1,2$,
the model describes diffusion processes with different nonlinear behavior
in the two subdomains. Transmission problems of this type have been
considered, for instance, in \cite{MolinoRossi}, where the Laplacian acts
in one subdomain and the $p$--Laplacian in the complementary region under
Dirichlet boundary conditions. Eigenvalue transmission problems with
Neumann boundary conditions driven by $(p_1,p_2)$-Laplacians were studied
in \cite{BMP}.

(iii) The variational framework developed in this paper also allows the case
$\Gamma_2=\emptyset$, which corresponds to a configuration where
$\Omega_2$ lies entirely inside $\Omega$ and its boundary coincides with the
interface $\Sigma$. In this case all boundary conditions on $\Gamma_2$
have to be dropped, while the transmission conditions on $\Sigma$ remain
unchanged (see \cite{BBM} for the case of $(p_1,p_2)$-Laplacians with a
Neumann–Robin boundary condition).

Transmission conditions of the form
\[
A_1(x,\nabla u_1)\cdot\nu_{1\Sigma}
=
A_2(x,\nabla u_2)\cdot\nu_{2\Sigma}
=
\mu\,(u_1-u_2)
\]
arise naturally in models describing the movement of individuals between
adjacent habitats separated by an interface, where the parameter $\mu>0$
measures the permeability of the interface (see \cite{AB}). In the limit
$\mu\to\infty$, perfect mixing across $\Sigma$ formally leads to the
classical transmission condition $u_1=u_2$ on $\Sigma$ considered here. In particular, the operators $A_i$ may represent
the Laplacian, corresponding to linear diffusion, or more general nonlinear
operators such as the $p_i$-Laplacian.

Other applications occur in models of electromagnetic wave propagation in
optical fibers, where transmission conditions appear at the interface
between the core and the cladding of the fiber (see \cite{Pf}). Related
transmission problems in domains partitioned into several subdomains were
studied in \cite{Ni}.

Nonlinear $p$-growth operators, also called $p$-Laplacian-like operators,
play an important role in the study of nonlinear elliptic and parabolic
equations. Such operators arise in several models from mathematical
physics, including capillarity phenomena and nonlinear elasticity
(see, e.g., \cite{FNELCCM, FNEL, CVetro, Vetro}).

When the exponent depends on the spatial variable one obtains the
$p(x)$-Laplacian, which arises in models of nonhomogeneous materials and
organic semiconductors (see \cite{Bul}). A particular case occurs when the
exponent $p(x)$ is piecewise constant, leading to diffusion operators that
change across different regions of the domain.

Regularity, existence, and multiplicity results for nonlinear elliptic
equations involving $p$-growth operators have been obtained under
various boundary conditions. For Dirichlet problems see \cite{PR, PRS}, while Robin problems were studied in \cite{PW, CPV}.

Variational existence results for elliptic transmission problems driven by
the Laplacian were obtained in \cite{FM12,FM}.
Eigenvalue transmission problems of Kirchhoff type were investigated in
\cite{Li, MM}.

In contrast with the single-domain case, the present setting requires a
modified variational approach due to the interaction between the
subdomains and the nonhomogeneous structure of the operators.

The main contributions of the paper are the following:
$(i)$~ the introduction of an appropriate variational framework for
problem~\eqref{eq:1.0}--\eqref{eq:1.00}; $(ii)$ the existence of an unbounded sequence of eigenvalues obtained
via genus theory and a Lusternik--Schnirelmann type argument;
$(iii)$ under suitable assumptions, the characterization of the
eigenvalue set as the entire interval $(0,\infty)$ using coercivity
arguments and the Mountain Pass Theorem;
$(iv)$ the corresponding eigenvalue results for the associated
single-domain problem.

The paper is organized as follows.
Section~\ref{sec:ar} introduces the functional framework and the assumptions,
fixes the notation, and contains several auxiliary results together with
the statements of the main theorems.
Sections~\ref{se:teorema1} and~\ref{sec:teorema2}
are devoted to the proofs of the main results.

\section{Preliminaries and main results}\label{sec:ar}

\subsection{Functional framework, assumptions, and main results}\label{ss:AFF}
We start by introducing the main notation and assumptions used throughout this paper.

Let $F:\mathbb{R}^N\to[0,\infty)$ be a fixed norm, strictly convex and of class
$C^1(\mathbb{R}^N\setminus\{0\})$.
Using a norm $F$ instead of the Euclidean norm $|\cdot|$ 
allows us to naturally include anisotropic operators.

We collect below the assumptions needed for the variational framework.

Let $p_1, p_2\in(1,\,\infty)$ be fixed. For $i\in\{1,2\}$, let $G_i:\Omega_i\times[0,\infty)\to[0,\infty)$ satisfy the following assumptions.
\begin{itemize}
\item[($H_{G1}$)]
For every $y\ge0$, the mapping $x\mapsto G_i(x,y)$ is measurable on $\Omega_i$
and for a.e. $x\in\Omega_i$, the function $y\mapsto G_i(x,y)$ is $C^1$,
strictly convex on $[0,\infty)$, and satisfies $G_i(x,0)=G_{iy}(x,0)=0$.
\item[($H_{G2}$)]
There exist  $d_i>0$ and $y_{0i}\ge0$ such that, for a.e.\ $x\in\Omega_i$,
\[
G_i(x,y)\ge d_i y^{p_i}
\quad \text{for all }y\ge y_{0i}.
\]
\item[($H_{G3}$)]
There exists a function $a_i\in L^\infty(\Omega_i)$ with $a_i\ge 0$  a.e.  on $\Omega_i$ such that
\[
G_{iy}(x,y)\le a_i(x)\big(1+y^{p_i-1}\big)
\quad \text{for a.e. }x\in\Omega_i\quad \text{and all } y\ge 0.
\]
\end{itemize}

\begin{remark}\label{re:exempleGF}
Let $D\in\{\Omega_1,\Omega_2\}$ be fixed.
Assume that $\kappa\in L^\infty(D)$ satisfies $\kappa(x)\ge \kappa_0>0$ a.e. on $D$.
Let $p>1$ be fixed, and let $F(\xi)=|\xi|$ be the Euclidean norm.
Typical examples of admissible functions $G:D\times[0,\infty)\to\mathbb{R}$ satisfying $(H_{G1})$-$(H_{G3})$  include
\begin{equation}\label{eq:g1}
G(x,y)=\frac{\kappa(x)}{p}\,y^{p}, 
\quad
G(x,y)=\frac{\kappa(x)}{p}y^{p}+\frac{\mu(x)}{r}y^{r},\quad 1<r<p,
\end{equation}
where $\mu\in L^\infty(D)$, $\mu\ge0$ a.e. in $D$; the corresponding operators are the weighted $p$-Laplacian and the weighted $(p,r)$-Laplacian.
Further examples with $p$-growth are
\begin{equation}\label{eq:g2}
G(y)=\frac{(k^2+y^2)^{\frac p2}-k^p}{p}, 
\quad
G(y)=\frac{\big(\sqrt{k^2+y^{2p}}-k\big)}{p},\quad k>0,
\end{equation}
which arise in regularized $p$-Laplacian and capillarity-type models (see \cite{Vetro}).
Examples from nonlinear elasticity include
\begin{equation}\label{eq:g3}
G(y)=\frac{(\sqrt{k^2+y^2}-k)^p}{p},\quad k>0
\end{equation}
(see \cite{FNELCCM,FNEL}), while logarithmic perturbations such as
\begin{equation}\label{eq:g4}
G(y)=\frac{1}{p}y^{p}+\frac{1}{r}\ln(1+y^{r}), 
\quad p\ge2,\quad 1<r\le p,
\end{equation}
are also admissible.

The class of admissible functions $F$ includes $\ell_r$-norms $F(\xi)=\|\xi\|_{r}$, $r>1$, which, for $G(y)=y^{p}/p$, 
lead to the classical $p$-Laplacian when $r=2$ and to the pseudo $p$-Laplacian when $r=p$ (see \cite{BF}),
as well as anisotropic norms of the form $F(\xi)=\sqrt{\langle M\xi,\xi\rangle}$,
where $M$ is a constant $N\times N$ symmetric positive definite matrix.
\end{remark}

\begin{remark}\label{re:esob}
For completeness, we recall the Sobolev and trace critical exponents. Let $D\subset\mathbb{R}^N$ be a bounded domain with smooth boundary and
$1<\theta<N$. Set
\[
\theta^*:=\frac{\theta N}{N-\theta},\quad
\theta_*:=\frac{\theta(N-1)}{N-\theta},
\]
which are the critical Sobolev and trace exponents, respectively. We set $\theta^*=\theta_*:=\infty$ whenever $\theta\ge N$.
The Sobolev embedding $W^{1,\theta}(D)\hookrightarrow L^{r}(D)$ is compact
for
(i) $1\le r<\theta^*$ if $1<\theta<N$;
(ii) $1\le r<\infty$ if $\theta=N$;
(iii) when viewed as an embedding into $C(\overline D)$ if $\theta>N$.
Similarly, the trace operator defines a compact embedding
$W^{1,\theta}(D)\hookrightarrow L^r(\partial D)$ for $1\le r<\theta_*$ when $1\le \theta< N$.
If $\theta\ge N$, the trace embedding is compact into $L^r(\partial D)$
for every $1\le r<\infty$.
See, e.g., \cite[Theorem~3.9.52]{DMP}, \cite[Section~6.10, 6.10.5]{KJF}.
\end{remark}

Before stating the main results, we introduce some additional notation
and assumptions.

\noindent
Let $\alpha_i \in L^{\infty}(\Omega_i)$ and $\beta_i \in L^{\infty}(\Gamma_i)$,
with $\alpha_i \ge 0$ a.e.  on $\Omega_i$ and $\beta_i \ge 0$ a.e.  on $\Gamma_i$.
Define
\begin{equation}\label{eq:defiq}
\begin{aligned}
\mathcal I_{\alpha\beta}
&:=\Big\{ m\in\{1,2\};~
\int_{\Omega_m}\alpha_m\,dx
+\int_{\Gamma_m}\beta_m\,d\sigma>0
\Big\},\\
q_{\min}&:=\min\{q_m;\ m\in\mathcal I_{\alpha\beta}\},
\qquad
q_{\max}:=\max\{q_m;\ m\in\mathcal I_{\alpha\beta}\},\\
p_{\min}&:=\min\{p_1,p_2\},
\qquad
p_{\max}:=\max\{p_1,p_2\}.
\end{aligned}
\end{equation}
\begin{itemize}
\item[$(H_{pq})$] For every $i\in\mathcal I_{\alpha\beta}$, we assume
$1<q_i<p_{i*}$ if $\beta_i\not\equiv 0$ on $\Gamma_i$, and
$1<q_i\le p_i^*$ otherwise.

\item[$(H_{\zeta\eta})$]  $1<\zeta_i <  p_{i}^*$ if $\rho_i\not\equiv 0$ in $\Omega_i$, and  $1<\eta_i < p_{i*}$ if  $\gamma_i\not\equiv 0$ on $\Gamma_i$.

\item[$(H_{\alpha\beta})$] $\alpha_i \in L^{\infty}(\Omega_i)$, $\beta_i \in L^{\infty}(\Gamma_i)$,
$\alpha_i \ge 0$ a.e.  on $\Omega_i$, $\beta_i \ge 0$ a.e.  on $\Gamma_i$.
Moreover, $\mathcal I_{\alpha\beta} \neq \emptyset$.

\item[$(H_{\rho\gamma})$]
$\rho_i \in L^{\infty}(\Omega_i)$,  $\gamma_i \in L^{\infty}(\Gamma_i)$.
Moreover, if $\zeta_i\ge p_{\min}$, then $\rho_i\ge 0$ a.e. on $\Omega_i$,
and if $\eta_i\ge p_{\min}$, then $\gamma_i\ge 0$ a.e. on $\Gamma_i$.
  \end{itemize}
For $\theta\in(1,\infty]$, we denote by $\|\cdot\|_{i\theta}$ the Lebesgue norm in
$L^\theta(\Omega_i)$, and by $\|\cdot\|_{\partial i\theta}$ the corresponding norm
in $L^\theta(\Gamma_i)$. 

A weak solution of problem \eqref{eq:1.0}--\eqref{eq:1.00} is a pair 
\(U=(u_1,u_2)\in W^{1,p_1}(\Omega_1)\times W^{1,p_2}(\Omega_2)\)
such that, for each \(i=1,2\), the function \(u_i\) satisfies the
corresponding equation in \(\Omega_i\) in the sense of distributions,
together with the associated boundary and transmission conditions
in the sense of traces.
Obviously, any solution $U=(u_1,u_2)$ of problem \eqref{eq:1.0}--\eqref{eq:1.00} can be identified with an element $u$ of the space
\begin{equation*}
W:=\left\{u\in W^{1,p_1}(\Omega);~ u\vert_{\Omega_2}\in W^{1,p_2}(\Omega_2)\right\},
\end{equation*}
where $u\vert_{\Omega_i}=u_i.$
We endow  $W$ with the norm
\begin{equation}\label{eq:1.nw}
\| u\|:=\| u_1 \|_1+\parallel u_2\|_2\quad\text{for all } u=(u_1, u_2)\in W,
\end{equation}
where $\| \cdot\|_i$ is defined by
$
\|u_i\|_i:= \| \nabla u_i \|_{ip_i}+\|u_i\|_{{ip_i}}.
$

The space $W$ defined above can be identified with the space
\begin{equation}\label{eq:1.2}
\mathcal{W}:= \big\{ (u_1,u_2)\in W^{1,p_1}(\Omega_1)\times W^{1,p_2}(\Omega_2);\,
u_1=u_2\, \text{on }\Sigma\,\text{in the sense of trace}\big\}
\end{equation}
 (see \cite[Remark~1.1]{BMP}). Hence, $\mathcal W$ is a reflexive Banach space, since it is a closed subspace of the
reflexive product $W^{1,p_1}(\Omega_1)\times W^{1,p_2}(\Omega_2)$. For simplicity of notation, we use the same symbol $\|\cdot\|$ for the norm on $W$ and on $\mathcal W$
\begin{definition}\label{de:sol}
We say that a real number $\lambda$ is an eigenvalue of problem~\eqref{eq:1.0}--\eqref{eq:1.00} 
if \eqref{eq:1.0}--\eqref{eq:1.00} admits a weak solution 
$\widetilde{u}_\lambda = (u_{1\lambda},u_{2\lambda}) \in \mathcal{W} \setminus \{(0,0)\}$. 
Such a function $\widetilde{u}_\lambda$ is called an eigenfunction corresponding
to $\lambda$, and the pair $(\lambda,\widetilde u_\lambda)$ is referred to as an
eigenpair of problem~\eqref{eq:1.0}--\eqref{eq:1.00}.
\end{definition}

For any $r>0$, we define the symmetric subset of $\mathcal{W}$
\begin{equation}\label{eq:2.defM}
\mathcal{M}_{r}=\Big\{\widetilde{u}=(u_1, u_2)\in \mathcal{W};~
\sum_{i=1}^2\frac{1}{q_i}\Big(\int_{\Omega_i}\alpha_i|u_i|^{q_i}\,dx
+ \int_{\Gamma_i}\beta_i|u_i|^{q_i}\,d\sigma\Big) = r \Big\}.
\end{equation}

We first establish the existence of an unbounded sequence of eigenvalues
via a Lusternik--Schnirelmann minimax principle, without any restriction on the
relative position of the exponents $q_i$ and $p_i$.

\begin{theorem}\label{teorema1}
Assume that  $(H_{G1})$, $(H_{G2})$, $(H_{G3})$, $(H_{pq})$, $(H_{\zeta\eta})$,
$(H_{\alpha\beta})$, and $(H_{\rho\gamma})$ are fulfilled.
Then, for every $r>0$, there exists a sequence of eigenpairs
$\big((\lambda_n,\pm(u_{1n},u_{2n}))\big)_n$ of problem~\eqref{eq:1.0}--\eqref{eq:1.00},
with $(u_{1n},u_{2n})\in\mathcal{M}_r$,
such that $\lambda_n\to\infty.$ 
\end{theorem}

To complement Theorem~\ref{teorema1}, we introduce the following additional assumptions.

\begin{itemize}

\item[($h_{G4}$)] 
There exist $\delta_i\in[0,p_i)$ and  $c_i>0$ such that for a.e. $x\in\Omega_i,$
\[
y\,G_{iy}(x,y)-p_i\,G_i(x,y)\le c_i y^{\delta_i}
\quad \text{for all }y\ge 0.
\]
\item[$(h_{\rho\gamma})$]
$\rho_i \in L^{\infty}(\Omega_i)$, $\gamma_i \in L^{\infty}(\Gamma_i)$, 
 $\rho_i \geq 0$ a.e. in $\Omega_i$, $\gamma_i \geq 0$ a.e. on $\Gamma_i$. 

Moreover, there exists $l\in\{1, 2\}$ such that
\[
\int_{\Omega_l} \rho_l\, dx +\int_{\Gamma_l} \gamma_l\, d\sigma >0.
\]
\end{itemize}

We now state the second main result of the paper, concerning two representative regimes determined by the growth exponents.
Other cases are left for future investigation.
\begin{theorem}\label{teorema2}
Let $\zeta_i=\eta_i=p_i$.
Assume that $(H_{G1})$, $(H_{G3})$, $(H_{pq})$,
$(H_{\alpha\beta})$, and $(h_{\rho\gamma})$ are fulfilled.
If either
\begin{itemize}
\item[(a)] the assumptions $(H_{G2})$ with $y_{0i}=0$ and $(h_{G4})$ hold and $p_{\max}<q_{\min}$, or
\item[(b)] $q_{\max}<p_{\min}$,
\end{itemize}
then the set of eigenvalues of problem~\eqref{eq:1.0}--\eqref{eq:1.00}
is $(0,\infty)$.
\end{theorem}

\begin{remark}\label{re:gh4examples}
All the functions in Remark~\ref{re:exempleGF} satisfy $(H_{G2})$ with
$y_{0i}=0$ and $(h_{G4})$, except for \eqref{eq:g3}, which satisfies $(H_{G2})$
with $y_{0i}>0$. However, the perturbed function
\[
G(y):=\frac{(\sqrt{k^2+y^2}-k)^p}{p}+\frac{y^p}{p}, \quad y\ge0,
\]
satisfies both assumptions with $\delta=p-1$ in $(h_{G4})$.

\noindent
An example satisfying $(H_{G1})$, $(H_{G2})$ with $y_{0i}=0$, $(H_{G3})$,
and $(h_{G4})$ with $\delta=p-\varepsilon$ is
\begin{equation}\label{eq:0G6}
G(y):=\frac{y^p}{p}\big(2-(1+y)^{-\varepsilon}\big),\quad y\ge 0,
\quad \varepsilon\in(0,p-1).
\end{equation}
Note that assumption $(h_{G4})$ does not follow from $(H_{G1})$--$(H_{G3})$.
For $p\ge2$, the function
\[
G(y):=\frac{y^p}{p}\Bigl(1+\frac{\ln(1+y)}{1+\ln(1+y)}\Bigr),\quad y\ge0,
\]
satisfies $(H_{G1})$--$(H_{G3})$ but not $(h_{G4})$. Therefore, Theorem~\ref{teorema2} does not apply in this example, whereas
Theorem~\ref{teorema1} still applies.
\end{remark}

As an immediate consequence of Theorems~\ref{teorema1} and~\ref{teorema2}, we obtain the following single-domain result.

\begin{corollary}[Single-domain case]\label{cor:singledomain}
Assume that the operators and coefficients coincide in the two subdomains,
that is, $p_1=p_2=:p$ and $G_1=G_2=:G$, and that the corresponding
exponents and weight functions are identical.
If assumptions $(H_{G1})$, $(H_{G2})$, $(H_{G3})$, $(H_{pq})$,
$(H_{\zeta\eta})$, $(H_{\alpha\beta})$, and $(H_{\rho\gamma})$
hold on $\Omega$, then for every $r>0$ there exists a sequence of eigenpairs
$\big((\lambda_n,\pm u_n)\big)_n$, with $u_n\in\mathcal{M}_r$, such that
$\lambda_n\to\infty$.

Moreover, if either
\begin{enumerate}
\item $(H_{G2})$ with $y_0=0$, $(h_{G4})$, and $(h_{\rho\gamma})$ hold on $\Omega$ and
$q>\max\{p,\zeta,\eta\}$, or
\item $(h_{\rho\gamma})$ holds on $\Omega$ and $q<\min\{p,\zeta,\eta\}$,
\end{enumerate}
then the set of eigenvalues of the corresponding single-domain problem
is $(0,\infty)$.
\end{corollary}

\subsection{Auxiliary results}

\begin{remark}\label{pr:PF}
The norm $F$ satisfies:
\begin{equation}\label{eq:Euler}
\nabla_\xi F (\xi)\cdot \xi = F (\xi)\quad\text{for all } \xi\in \mathbb{R}^N,\quad\xi \neq 0.
\end{equation}
\begin{equation}\label{eq:grad}
\nabla_\xi F (t\xi) = \operatorname{sgn}(t) \nabla_\xi F (\xi)\quad\text{for all }
 \xi \neq 0\quad\text{and  all }t \neq 0.
\end{equation}
\begin{equation}\label{eq:equiv}
m |\xi| \leq F (\xi) \leq M |\xi|
\quad\text{for all } \xi \in \mathbb{R}^N,
\end{equation}
for some positive constants $m, M$.
\end{remark}
For convenience, we define $\widetilde{G}_i:\Omega_i\times\mathbb{R}^N\to\mathbb{R}$ by
\begin{equation}\label{eq:01.n}
\widetilde{G}_i(x, z)
:= G_i\big(x,F(z)\big),\quad
(x,z)\in\Omega_i\times\mathbb{R}^N.
\end{equation}
For $\widetilde{u}=(u_1, u_2)\in \mathcal{W}$,  we set
\begin{equation}\label{eq:1.n}
\begin{aligned}
K_{p_1p_2}(\widetilde{u})&:=\sum_{i=1}^2 \int_{\Omega_i}G_i(x,  F(\nabla u_i(x)))\,dx,\\
k_{\zeta\eta}(\widetilde{u})&:=\sum_{i=1}^2 \frac{1}{\zeta_i}\int_{\Omega_i}\rho_i|u_i|^{\zeta_i}\,dx
+ \sum_{i=1}^2 \frac{1}{\eta_i}\int_{\Gamma_i}\gamma_i |u_i|^{\eta_i}\,d\sigma,\\
\mathcal{H}_{q_1q_2}(\widetilde{u})&:=\sum_{i\in \mathcal I_{\alpha\beta}}\frac{1}{q_i}\left(\int_{\Omega_i}\alpha_i|u_i|^{q_i}\,dx
+ \int_{\Gamma_i}\beta_i| u_i|^{q_i}\,d\sigma\right).
\end{aligned}
\end{equation}
\begin{remark}\label{re:estutile}
We collect below some estimates which will be used repeatedly in the sequel.

By $(H_{G3})$, together with \eqref{eq:grad} and \eqref{eq:equiv}, we obtain
\begin{equation}\label{eq:estgi}
\begin{aligned}|\nabla_\xi& \widetilde{G}_i(x, z)| = | {G}_{iy}(x, F(z)) \nabla_\xi F(z)|\le C_i a_i(x)\big(1 + |F(z)|^{p_i-1}\big)\\
&\le \overline a_i(x)\big(1 + \ |z|^{p_i-1}\big)\quad\text{for a.e. }x\in \Omega_i \quad\text{and all }z\in\mathbb{R}^N,\quad z\neq 0,
\end{aligned}
\end{equation}
where $\overline{a}_i(x):=a_i(x)C_i\max\,\{1, M^{p_i-1}\}$ and $C_i := \max_{F(\xi)=1}|\nabla F(\xi)| > 0.$

In addition, by $(H_{G2})$, there exists $d_i>0$ such that
\begin{equation}\label{eq:HG2nou}
G_i(x,y)\ge d_i y^{p_i}- D_i
\quad\text{for a.e. }x\in\Omega_i\quad\text{and all }y\ge 0,
\end{equation}
where $D_i=d_i  y_{0i}^{p_i}.$ Moreover, integrating over $\Omega_i$ and combining with \eqref{eq:equiv},
this leads to
\begin{equation}\label{eq:mgK}
\begin{aligned}
K_{p_1p_2}(\widetilde u)
&=\sum_{i=1}^2\int_{\Omega_i}G_i\big(x, F(\nabla u_i)\big)\,dx\ge \Big(\sum_{i=1}^2 d_i \int_{\Omega_i}F(\nabla u_i)^{p_i}\,dx-C_{2i}\Big)\\
&\ge C_{1}\sum_{i=1}^2\|\nabla u_i\|_{ip_i}^{p_i}-C_{2}\quad\text{for all }\widetilde u=(u_1,u_2)\in\mathcal W,
\end{aligned}
\end{equation}
where $C_1:=\min\{d_1 m^{p_1}, d_2 m^{p_2}\}>0 $ and $C_2:=\sum_{i=1}^2 d_i  y_{0i}^{p_i}\operatorname{meas}\Omega_i.$ 

Finally, by the convexity of $y\mapsto G_i(x,y)$ on $[0,\infty)$
and the condition $G_i(x,0)=0$ in $(H_{G1})$, we have
\begin{equation}\label{eq:yGy}
yG_{iy}(x,y)\ge G_i(x,y)\quad\text{for a.e. } x\in\Omega_i\quad\text{and all }y\ge 0.
\end{equation}
\end{remark}
We define the functional $\mathcal{J}:\mathcal{W}\to\mathbb{R}$ by
$
\mathcal{J}(\widetilde{u})
:= K_{p_1p_2}(\widetilde{u}) + k_{\zeta\eta}(\widetilde{u}),
~ \widetilde{u}\in \mathcal{W}.$

Under assumptions $(H_{G1})$, $(H_{G3})$, and $(H_{\zeta\eta})$,
the functional $\mathcal{J}$ is weakly lower semicontinuous on
$\mathcal W$ and belongs to $C^1(\mathcal W; \mathbb R)$, with derivative given by 
\begin{equation}\label{eq:derivataJ}
\begin{aligned}
\langle \mathcal{J}'\big((u_1, u_2)\big),(h_1, h_2)\rangle
 &=\sum_{i=1}^2\Big(
    \int_{\Omega_i}
    G_{iy} \big(x,F(\nabla u_i) \big)\nabla_\xi F(\nabla u_i)\cdot\nabla h_i\,dx \\
   &\quad +\int_{\Omega_i}\rho_i |u_{i}|^{\zeta_i-2}u_{i} h_i\,dx
    +\int_{\Gamma_i}\gamma_i |u_{i}|^{\eta_i-2}u_{i} h_i\,d\sigma    \Big).
\end{aligned}
\end{equation}
The $C^1$-regularity of $\mathcal J$ follows from standard arguments (see, e.g., \cite[Section~26.5]{ZIIB} and \cite[Lemma~2.16]{Wi}).
Since $G_i$ is convex in the second argument, the functional
$K_{p_1p_2}$ is weakly lower semicontinuous on $\mathcal W$ (see \cite[Corollary 3.9]{Br}).
Moreover, the compact embeddings in Remark~\ref{re:esob} imply strong convergence of the potential terms;
hence $\mathcal J$ is weakly lower semicontinuous.

By Definition~\ref{de:sol}, $\lambda$ is an eigenvalue of
\eqref{eq:1.0}--\eqref{eq:1.00} if and only if there exists
$\widetilde{u}_\lambda=(u_{1\lambda}, u_{2\lambda})\in \mathcal W\setminus\{0\}$, such that
 for all $(v_1,v_2)\in\mathcal W$,
\begin{equation}\label{eq:1.def}
\langle \mathcal J'(\widetilde{u}_\lambda),(v_1,v_2)\rangle
=
\lambda \sum_{i\in \mathcal I_{\alpha\beta}}
\Big(
\int_{\Omega_i}\alpha_i |u_{i\lambda}|^{q_i-2}u_{i\lambda} v_i\,dx
+\int_{\Gamma_i}\beta_i |u_{i\lambda}|^{q_i-2}u_{i\lambda} v_i\, d\sigma
\Big).
\end{equation}
The above variational characterization of the eigenvalues of problem
\eqref{eq:1.0}--\eqref{eq:1.00} can be obtained by arguments similar to those in \cite[Proposition~1.1]{BMP}.

We recall the definition of the Krasnosel'skii genus.

Let $X$ be a real Banach space and denote by $\mathcal E$ the family of all nonempty
closed and symmetric subsets of $X\setminus\{0\}$.
For $A\in\mathcal E$, the Krasnosel'skii genus $\gamma(A)$ is defined as the smallest nonegative
integer $n$ such that there exists an odd continuous mapping
$A\mapsto\mathbb{R}^n\setminus\{0\}$. If $A=\emptyset$, we set $\gamma(A)=0$, while
$\gamma(A)=\infty$ if no such finite $n$ exists.

We also recall the Palais--Smale condition on a constraint manifold
(see, e.g., \cite[Definition 44.13]{Z}).
\begin{definition}\label{de:dif}
Let $\mathbf M\subset X$, and
$f: X\to\mathbb R$ be a functional admitting a tangential derivative with respect to $\mathbf M$,
$f'_{\mathbf M},$ at each point of $\mathbf M$.
For $c\in\mathbb R$, we say that $f$ satisfies the  Palais--Smale condition
$(PS)_c$ with respect to $\mathbf M$  if every sequence $\big(u_n\big)_n\subset\mathbf M$ such that
$
f(u_n)\to c$ and $\|f'_{\mathbf M}(u_n)\|_{(T_{u_n}\mathcal M_r)^*}\to 0$
has a convergent subsequence in $X$.
We say that $f$ satisfies the Palais--Smale condition on $\mathbf M$ 
if and only if $(PS)_c$ is satisfied for all $c\in \mathbb R$.
\end{definition}
We conclude this section with a Lusternik--Schnirelmann type result on $C^1$-manifolds (see \cite[Corollary~4.1, p.~132]{Sz}).
\begin{theorem}\label{te:Sz}
Let $\mathbf M$ be a closed symmetric $C^1$-submanifold of a real Banach space $X$
such that $0\notin\mathbf M$, and let $f\in C^1(\mathbf M; \mathbb{R})$ be even and
bounded from below. For $j\in\mathbb{N}$, define
$c_j := \inf_{A\in\Lambda_j}\sup_{x\in A} f(x)$
where
$
\Lambda_j=\{A\subset\mathbf M;~ A\in\mathcal E,~ \gamma(A)\ge j,~ \text{and }A
\text{ is compact}\}.
$
If $\Lambda_k\neq\emptyset$ for some $k\ge1$ and if $f$ satisfies the
Palais--Smale condition at all levels $c_j$, $j\in\{1,\dots,k\}$, then $f$ has
at least $k$ distinct pairs of critical points.
\end{theorem}

\section{ Proof of Theorem~\ref{teorema1}}\label{se:teorema1}

The proof of Theorem~\ref{teorema1} will be derived from a sequence of intermediate results. 
\begin{lemma}\label{le:ig}
Let $r>0$. Assume that $(H_{\alpha\beta})$ and $(H_{pq})$ are fulfilled.
Then  $\gamma(\mathcal M_r)=\infty$ .
\end{lemma}

\begin{proof}
By assumption $(H_{\alpha\beta})$, there exists $i_0\in\{1,2\}$ such that $i_0\in \mathcal I_{\alpha\beta},$ i.e.,
\[
\int_{\Omega_{i_0}}\alpha_{i_0}\,dx+\int_{\Gamma_{i_0}}\beta_{i_0}\,d\sigma>0.
\]
Set $
E:=\{x\in\Omega_{i_0};~\alpha_{i_0}(x)>0\},~F:=\{x\in\Gamma_{i_0};~\beta_{i_0}(x)>0\}.
$
Hence, either $\operatorname{meas}(E)>0$ or $\mathcal{H}_{N-1}\,(F)>0.$
We treat the bulk case; the boundary case is similar.

Fix $k\in\mathbb N^*$.
Since $\operatorname{meas}(E)>0$, we may choose $x_1,\dots,x_k\in E$ and 
$\delta_1,\dots,\delta_k>0$ such that the balls $B_{\delta_j}(x_j)\Subset\Omega_{i_0}$,
are pairwise disjoint, and
$\operatorname{meas}(E\cap B_{\delta_j}(x_j))>0$ for every $j\in\{1,\dots,k\}$.
For each fixed $j$, let $\varphi_j\in C_c^\infty(\Omega_{i_0})$ be defined by
\[
\varphi_j(x)=
\begin{cases}
e^{ -\frac{1}{\delta_j^2-|x-x_j|^2}}, & x\in B_{\delta_j}(x_j),\\
0,& \text{otherwise}.
\end{cases}
\]
We have
\[
\theta_j:=\frac1{q_{i_0}}\int_{\Omega_{i_0}}\alpha_{i_0}|\varphi_j|^{q_{i_0}} dx  >0, \quad
\widetilde\varphi_j:=\Big(\big(\tfrac{r}{\theta_j}\big)^{1/q_{i_0}}\varphi_j,\,0\Big)\in\mathcal W,
\]
and $\mathcal{H}_{q_1q_2}(\widetilde\varphi_j)=r$, hence $\widetilde\varphi_j\in\mathcal M_r$. 

Let $V_k:=\text{span}\{\widetilde\varphi_1,\dots,\widetilde\varphi_k\}$.
The family $\{\widetilde \varphi_1,\dots,\widetilde \varphi_k\}$ is linearly independent,
as their supports are pairwise disjoint, hence $\dim V_k=k$. 

On the other hand, for
$\widetilde u=\sum_{j=1}^k t_j\widetilde\varphi_j\in V_k$, we have
$
\mathcal{H}_{q_1q_2}(\widetilde u)=r\sum_{j=1}^k |t_j|^{q_{i_0}}.
$
Therefore,
\[
\mathcal M_r\cap V_k
=\Big\{\sum_{j=1}^k t_j\widetilde\varphi_j;~ \sum_{j=1}^k |t_j|^{q_{i_0}}=1\Big\}
\]
is homeomorphic to $S^{k-1}$.
Hence,
$
\gamma(\mathcal M_r)\ge \gamma(\mathcal M_r\cap V_k)=\gamma(S^{k-1})=k$
(see, e.g., \cite[Corollary~1.3]{Ra74}).
Since $k$ is arbitrary, $\gamma(\mathcal M_r)=\infty$.
\end{proof}
Due to the transmission condition, an equivalent norm can be defined on $\mathcal W$
even in the case where the weights multiplying the parameter
$\lambda$ vanish on one subdomain.
\begin{lemma}\label{le:normatransmission}
Assume that $(H_{pq})$ and $(H_{\alpha\beta})$ are fulfilled. If $j\in \mathcal I_{\alpha\beta}$,  then 
\begin{equation}\label{eq:normahatl}
\|\widetilde u\|_{j\alpha\beta}
:=\sum_{i=1}^2\|\nabla u_i\|_{ip_i} +s_{j\alpha\beta}(u_j),
\quad
\widetilde u=(u_1,u_2)\in\mathcal W,
\end{equation}
defines a norm on $\mathcal W$ which is equivalent to the norm in~\eqref{eq:1.nw}.
Here
\begin{equation}\label{eq:seminorme}
s_{j\alpha\beta}(u_j)
:=\Big(
\int_{\Omega_j}\alpha_j |u_{j}|^{q_j}\,dx
+\int_{\Gamma_j}\beta_j |u_j|^{q_j}\,d\sigma
\Big)^{1/q_j}.
\end{equation}
In addition, 
\begin{equation}\label{eq:lalbe}
\|\widetilde u\|_{j\alpha\beta}
\le \sum_{i=1}^2\|\nabla u_i\|_{ip_i}
+  (q_{j} r)^{1/q_{j}}
\quad \text{for all } \widetilde u\in\mathcal M_r.
\end{equation}
\end{lemma}

\begin{proof}
The inequality
$
\|\widetilde u\|_{j\alpha\beta} \le C_1 \|\widetilde u\|$
follows from  Remark~\ref{re:esob}.
It remains to prove the converse estimate.

Suppose by contradiction that there exists a sequence
$\widetilde u_n=\big(u_{1n},u_{2n}\big)_n \subset \mathcal W\setminus \{0\}$ such that $\|\widetilde u_n\|\ge n\|\widetilde u_n\|_{j\alpha\beta}$ for all $n\in\mathbb N.$ Hence, we may assume that
\begin{equation}\label{eq:ww}
\|\widetilde u_n\|=1
~\text{and }
\|\nabla u_{1n}\|_{1p_1}
+\|\nabla u_{2n}\|_{2p_2}
+s_{j\alpha\beta}(u_{j n}) \to 0\quad\text{as }n\to\infty.
\end{equation}
Fix $i\in\{1,2\}$. By \eqref{eq:ww} we have $\|\nabla u_{in}\|_{ip_i}\to 0$.
Since $(\widetilde u_n)_n$ is bounded in $\mathcal W$, up to a subsequence,
$u_{in}\rightharpoonup u_i ~ \text{in } W^{1, p_i}(\Omega_i).$

By weak lower semicontinuity of norm,
$\|\nabla u_i\|_{ip_i}\le \liminf_{n\to\infty}\|\nabla u_{in}\|_{ip_i}=0$,
hence $u_i$ is constant on  $\Omega_i$.
Moreover, the transmission condition leads to $u_1=u_2$ on  $\Sigma$,
so there exists $c\in\mathbb R$ such that
$u_i= c$ in $\Omega_i$.

Using $(H_{pq})$ and Remark~\ref{re:esob}, we infer that
$
s_{j\alpha\beta}(u_{jn}) \to s_{j\alpha\beta}(u_j).
$
Passing to the limit in \eqref{eq:ww} gives $s_{j\alpha\beta}(u_j)=0$.
By assumption $j\in \mathcal I_{\alpha\beta}$,
this implies $c=0$.

Therefore, $u_i= 0$ in $\Omega_i$.
Since $\nabla u_{in}\to 0$ in $L^{p_i}(\Omega_i)$ and
$u_{in}\to 0$ strongly in $L^{p_i}(\Omega_i)$,
we conclude that $\widetilde u_n=(u_{1n}, u_{2n})\to(0, 0)$ in $\mathcal W$,
contradicting $\|\widetilde u_n\|=1$. 

Inequality \eqref{eq:lalbe} follows from \eqref{eq:normahatl} and the definition of $\mathcal M_r.$
\end{proof}

\begin{remark}\label{re:esobW}
Since $\mathcal W$ is a closed subspace of
$W^{1,p_1}(\Omega_1)\times W^{1,p_2}(\Omega_2)$, Remark~\ref{re:esobW} yields the compact embeddings
$\mathcal W\hookrightarrow L^{\theta_1}(\Omega_1)\times L^{\theta_2}(\Omega_2)$ and
$\mathcal W\hookrightarrow L^{\theta_1}(\Gamma_1)\times L^{\theta_2}(\Gamma_2)$
for $\theta_i<p_i^*$ and $\theta_i<p_{i*}$, respectively.
\end{remark}
In what follows, $C_1, C_2, \dots$ and $C_{1i}, C_{2i}, \dots$ denote positive
constants related to $\Omega$ and to $\Omega_i$, respectively, independent of
the variables under consideration (such as $\widetilde u$, $n$, etc.).

We have the following auxiliary result concerning the coercivity of $\mathcal{J}$.
\begin{lemma}\label{le:coerciv}
Let $r>0.$ Under the assumptions of Theorem~\ref{teorema1}, the functional $\mathcal{J}$ is coercive on $\mathcal{M}_r$; i.e.,
$\mathcal{J}(\widetilde u)\to\infty$ as $\|\widetilde u\|\to\infty$ for $\widetilde u\in\mathcal{M}_r.$
\end{lemma}

\begin{proof}
From estimate~\eqref{eq:mgK} for $K_{p_1p_2}$ we obtain
\begin{equation}\label{eq:Jcoercmr}
\mathcal J(\widetilde u)
\ge C_1\sum_{i=1}^2\|\nabla u_i\|_{ip_i}^{p_i}
+k_{\zeta\eta}(\widetilde u)-C_2\quad\text{for all }\widetilde{u}\in \mathcal W.
\end{equation}
Assume by contradiction that $\mathcal J$ is not coercive on $\mathcal M_r$.
Then there exists a sequence $\big(\widetilde u_n\big)_n\subset\mathcal M_r$ such that
$\|\widetilde u_n\|\to\infty$ while $\big(\mathcal J(\widetilde u_n)\big)_n$ is bounded.

From $(H_{\alpha\beta})$, there exists $i_0\in \mathcal I_{\alpha\beta}.$
Hence, by Lemma~\ref{le:normatransmission} (with $j$ replaced by $i_0$),
the norms $\|\cdot\|$ and $\|\cdot\|_{i_0\alpha\beta}$ are equivalent on $\mathcal W$. 
Moreover, since $\big(\widetilde u_n\big)_n\subset\mathcal M_r$, estimate~\eqref{eq:lalbe} with $j$ replaced by $i_0$
holds for $\widetilde u_n$. 
Therefore, $\|\widetilde u_n\|_{i_0\alpha\beta}\to\infty$ implies that
\begin{equation}\label{eq:div}
\sum_{i=1}^2\|\nabla u_{in}\|_{ip_i}\to\infty\quad \text{as }n\to \infty.
\end{equation}
Next, we derive an estimate for $k_{\zeta\eta}(\widetilde u).$

Fix $i\in\{1, 2\}$. Assume first that  $\zeta_i<p_{\min}$.
Using  Remark~\ref{re:esobW}, we obtain
\[
\Big|\int_{\Omega_i}\rho_i|u_{in}|^{\zeta_i}\,dx\Big|
\le C_{3i}\|\widetilde u_n\|_{\mathcal W}^{\zeta_i}
\le C_{4i}\|\widetilde u_n\|_{i_0\alpha\beta}^{\zeta_i}\quad\text{for all }n\in\mathbb N.
\]
Since $\big(\widetilde u_n\big)_n\subset\mathcal M_r$, estimate~\eqref{eq:lalbe} implies
\[
\|\widetilde u_n\|_{i_0\alpha\beta}
\le \sum_{i=1}^2\|\nabla u_{in}\|_{ip_i} + C_5\quad\text{for all }n\in\mathbb N.
\]
Therefore,  for all $n\in\mathbb N$, we have
\begin{equation}\label{eq:estimaripotentiali}
\Big|\int_{\Omega_i}\rho_i|u_{in}|^{\zeta_i}\,dx\Big|
\le C_{4i}\Big(\sum_{i=1}^2\|\nabla u_{in}\|_{ip_i} + C_5\Big)^{\zeta_i}
\le C_{6i}\Big(1+\sum_{i=1}^2\|\nabla u_{in}\|_{ip_i}^{\zeta_i}\Big).
\end{equation}
If $\zeta_i\ge p_{\min}$, then by $(h_{\rho\gamma})$ we have
$\int_{\Omega_i}\rho_i |u_{in}|^{\zeta_i}\,dx\ge 0~\text{for all }n\in\mathbb N$.

The same estimates hold for the boundary term with $\eta_i$.
Consequently,
\[
k_{\zeta\eta}(\widetilde u_n)
\ge
-C_7\Big(1+\sum_{i=1}^2\big(\|\nabla u_{in}\|_{ip_i}^{\theta_i}+\|\nabla u_{in}\|_{ip_i}^{\mu_i}\big)\Big)\quad\text{for all }n\in\mathbb N,
\]
with $\theta_i,\mu_i\in[0,p_{\min})$ (namely, $\theta_i=\zeta_i$ if $\zeta_i<p_{\min}$ and $\theta_i=0$ otherwise,
and $\mu_i=\eta_i$ if $\eta_i<p_{\min}$ and $\mu_i=0$ otherwise).
Inserting this into \eqref{eq:Jcoercmr}, we derive
\begin{equation}\label{eq:CC}
\mathcal J(\widetilde u_n)
\ge
C_1\sum_{i=1}^2\|\nabla u_{in}\|_{ip_i}^{p_i}
-C_7\Big(1+\sum_{i=1}^2\big(\|\nabla u_{in}\|_{ip_i}^{\theta_i}+\|\nabla u_{in}\|_{ip_i}^{\mu_i}\big)\Big)-C_2.
\end{equation}
By \eqref{eq:div}, there exists $l\in\{1,2\}$ such that, up to a subsequence,
$\|\nabla u_{ln}\|_{lp_l}\to\infty$. Since $\max\{\theta_l,\mu_l\}<p_l$,
the right-hand side in \eqref{eq:CC} tends to infinity, which contradicts the
boundedness of $\big(\mathcal J(\widetilde u_n)\big)_n$.
\end{proof}

A proof of the next result in a general subdifferential setting is given in
\cite[Proposition~2]{PRS}. We include here a proof adapted to our variational
framework.
\begin{lemma}\label{le:in}
Assume that $(H_{G1})$, $(H_{G2})$, and $(H_{G3})$ are fulfilled. 
Then, for any weakly
convergent sequence $(\widetilde u_n)_n\subset\mathcal W$ with weak limit $\widetilde u_*$ for which  
\begin{equation}\label{eq:convK}
\big\langle K_{p_1p_2}'(\widetilde u_n),\,\widetilde u_n-\widetilde u_*\big\rangle \to 0\quad\text{as } n\to \infty,
\end{equation}
we have, $\widetilde u_n\to \widetilde u_*$ in $\mathcal W$.
\end{lemma}

\begin{proof}
Let $(\widetilde u_n)_n\subset\mathcal W$ be such that $\widetilde u_n \rightharpoonup \widetilde u_*$  in $\mathcal W$ and \eqref{eq:convK} holds. 
In particular, we have
$\big\langle K_{p_1p_2}'(\widetilde{u}_n)-K_{p_1p_2}'(\widetilde{u}_*), \widetilde{u}_n-\widetilde{u}_*\big\rangle \to 0.
$
Equivalently, 
\begin{equation}\label{eq:1n}
\sum_{i=1}^2\int_{\Omega_i}
\big(\nabla_\xi  \widetilde{G}_i(x, \nabla u_{in})-\nabla_\xi  \widetilde{G}_i(x,\nabla u_{i*})\big)\cdot\nabla(u_{in}-u_{i*})\,dx \to 0\quad\text{as } n\to \infty,
\end{equation}
where $\widetilde{u}_n=(u_{1n},u_{2n})$ and $\widetilde{u}_*=(u_{1*},u_{2*})$. 

\noindent
To prove the strong convergence $\widetilde u_n\to \widetilde u_*$ in $\mathcal W$, it suffices to verify 
that every such sequence admits a strongly convergent subsequence converging to  $\widetilde u_*$.

Fix $i\in\{1, 2\}$. From Remark~\ref{re:esobW}
we already know that, up to a subsequence, 
$
u_{in}\to u_{i*} 
$ in $L^{p_i}(\Omega_i).$
Thus, in order to complete the proof it remains to verify that
\begin{equation}\label{eq:convgrad}
\nabla u_{in}\to \nabla u_{i*}\quad \text{in } L^{p_i}(\Omega_i).
\end{equation}
To this end, we apply the Vitali Convergence Theorem (see, e.g., \cite[Ex.~15, p.~187]{Fo}).

We first establish that $\nabla u_{in}\to \nabla u_{i*}$ a.e. on $\Omega_i$.  For $x \in \Omega_i$,  we set
\begin{equation}\label{eq:0n}
\Gamma_{in}(x):= \big(\nabla_\xi  \widetilde{G}_i(x, \nabla u_{in}(x))-\nabla_\xi  \widetilde{G}_i(x,\nabla u_{i*}(x))\big)\cdot\nabla (u_{in}(x)-u_{i*}(x)).
\end{equation}
From  $(H_{G1})$, the mapping $\xi \to \widetilde{G}_i (x, \xi)$  is strictly convex on $\mathbb{R}^N$ for a.e. $x\in \Omega_i$, which implies that
$\Gamma_{in}(x)\ge 0$ a.e. on $\Omega_i$. Hence, from \eqref{eq:1n} we obtain that
$\Gamma_{in}\to 0$ in $L^1(\Omega_i)$.

Consequently,  up to a subsequence, $\Gamma_{in}(x)\to 0$ a.e. on $\Omega_i.$ 
Let $\widehat{\Omega}_i$  be a measurable subset of $\Omega_i$ such that
$\operatorname{meas} (\Omega_i\setminus\widehat{\Omega}_i)=0$,
the  convergence $\Gamma_{in}(x)\to 0$ holds on $\widehat{\Omega}_i$
and $\nabla u_{i*}(x)$ is finite everywhere on $\widehat{\Omega}_i$.

On the one hand, by \eqref{eq:HG2nou} and \eqref{eq:equiv},
we obtain
\begin{equation}\label{eq:3n}
\nabla_\xi \widetilde{G}_i(x,\xi)\cdot \xi  \ge  \widetilde{G}_i(x,\xi)  \ge m d_i |\xi|^{p_i} -D_i\quad\text{for a.e. }x\in\Omega_i\quad \text{and all }\xi\in\mathbb{R}^N.
\end{equation}
On the other hand, from \eqref{eq:estgi}, we have
\begin{equation}\label{eq:1nou}
|\nabla_\xi \widetilde{G}_i(x, \xi)| \le \overline{a}_i(x)\big(1 + |\xi|^{p_i-1}\big)\quad\text{for a.e. }x\in\Omega_i\quad \text{and all }\xi\in\mathbb{R}^N.
\end{equation}
Hence, by \eqref{eq:3n} and \eqref{eq:1nou}, for $x\in\widehat{\Omega}_i$, we have
\begin{equation*}
\begin{split}
\Gamma_{in}(x) \ge m d_i |\nabla u_{in}(x)|^{p_i} - M_{1i}(x) |\nabla u_{in}(x)|- M_{2i}(x) |\nabla u_{in}(x)|^{p_i-1}-M_{3i}(x), 
\end{split}
\end{equation*}
where the functions $M_{1i}, M_{2i}$, and $M_{3i}$ may also be assumed to
be finite on $\widehat{\Omega}_i.$

This implies that $\big(\nabla u_{in}(x)\big)_n$ is bounded in $\mathbb{R}^N$, for each $x\in\widehat\Omega_i$.

Let $x\in\widehat\Omega_i$ be fixed. There exists a
subsequence of $\big(\nabla u_{in}(x)\big)_n$ which converges to a limit $v_i(x)$.
Then, along the same subsequence,
passing to the limit in $\Gamma_{in}(x)\to 0,$ 
since
$\xi\to \nabla_\xi \widetilde G_i(x,\xi)$ is strictly monotone on $\mathbb{R}^N$, we get $v_i(x)=\nabla u_{i*}(x)$.
Thus, $\nabla u_{in}(x)\to \nabla u_{i*}(x)$ a.e. on $\Omega_i.$ 

Next, we verify that $\big(|\nabla u_{in}|^{p_i}\big)_n$ is uniformly integrable.

\noindent
Expanding the scalar product in the definition of $\Gamma_{in}(x)$ and using \eqref{eq:3n} and \eqref{eq:1nou}, we obtain
\begin{equation}\label{eq:ui}
\begin{aligned}
m d_i|\nabla u_{in}(x)|^{p_i}&\le b_i(x)+ \Gamma_{in}(x) + \overline{a}_i(x)\big(1+|\nabla u_{i*}(x)|^{p_i-1}\big)|\nabla u_{in}(x)|\\
&+\overline{a}_i(x)\big(1+|\nabla u_{in}(x)|^{p_i-1}\big)|\nabla u_{i*}(x)|,
    \end{aligned}
\end{equation}
where $b_i\in L^1(\Omega_i)$. Each of the terms on the right-hand side of the
inequality is uniformly integrable. Indeed, for $b_i$, this is obvious. For the second term this follows from  $\Gamma_{in}\to 0$ in $L^1(\Omega_i)$.  
Finally, for the last two terms we have
\begin{equation*}
\begin{split}
\int_E \overline{a}_i\big(1+|\nabla u_{i*}|^{p_i-1}\big)&|\nabla u_{in}|\,dx \le \|\overline{a}_i\|_{i\infty}\|1+|\nabla u_{i*}|^{p_i-1}\|_{L^{\overline{p}_i}(E)} \|\nabla u_{in}\|_{L^{p_i}(E)}\\
&\le C_1\|1+|\nabla u_{i*}|^{p_i-1}\|_{L^{\overline{p}_i}(E)},\\
\int_E \overline{a}_i\big(1+|\nabla u_{in}|^{p_i-1}\big)&|\nabla u_{i*}|\,dx \le \|\overline{a}_i\|_{i\infty}\|1+|\nabla u_{in}|^{p_i-1}\|_{L^{\overline{p}_i}(E)} \|\nabla u_{i*}\|_{L^{p_i}(E)}\\
&\le C_2\|\nabla u_{i*}\|_{L^{p_i}(E)}\quad\text{for all measurable set } E\subset\Omega_i,
    \end{split}
\end{equation*}
 with $1/p_i+1/\bar{p}_i=1$. Since $1+|\nabla u_{i*}|^{p_i-1}\in L^{\bar{p}_i}(\Omega_i)$ and $\nabla u_{i*}\in L^{p_i}(\Omega_i)$,
these last two estimates ensure the uniformly integrability of the last two terms in \eqref{eq:ui}. Hence $\big(|\nabla u_{in}|^{p_i}\big)_n$ is uniformly integrable on $\Omega_i$. 

As $\Omega_i$ has finite measure, combining  $\nabla u_{in}(x)\to \nabla u_{i*}(x)$ a.e. on $\Omega_i$
with the uniform integrability of the sequence $\big(|\nabla u_{in}|^{p_i}\big)_n$,
Vitali's Convergence Theorem implies $\|\nabla u_{in}\|_{ip_i}\to \|\nabla u_{i*}\|_{ip_i}.$
Since $\nabla u_{in}\rightharpoonup \nabla u_{i*}$ in $L^{p_i}(\Omega_i;\mathbb R^N)$ and
$L^{p_i}(\Omega_i;\mathbb R^N)$ is uniformly convex, it follows that
$\nabla u_{in}\to \nabla u_{i*}$ strongly in $L^{p_i}(\Omega_i;\mathbb R^N)$, hence \eqref{eq:convgrad}.
This concludes the proof.
\end{proof}

Next, from \eqref{eq:1.def}, $\lambda$ is an eigenvalue of
problem~\eqref{eq:1.0}--\eqref{eq:1.00} if and only if there exists
$\widetilde u_\lambda\in \mathcal W\setminus\{(0,0)\}$ such that
\begin{equation}\label{eq:1.defech}
\mathcal J'(\widetilde u_\lambda)
= \lambda \mathcal H'_{q_1q_2}(\widetilde u_\lambda).
\end{equation}
Since for all $\widetilde{u}\in\mathcal{M}_r$ one has 
$\langle \mathcal{H}_{q_1q_2}^\prime (\widetilde{u}), \widetilde{u}\rangle \neq 0,$
it follows that $r$ is a regular value of the $C^1$-functional $\mathcal{H}_{q_1q_2}$. 
Therefore, $\mathcal{M}_r=\mathcal{H}_{q_1q_2}^{-1}(r)$ is a $C^1$-symmetric manifold of codimension $1$ in $\mathcal{W}$ 
(see, e.g., \cite[Theorem~2.2.7]{PK}), with tangent space at 
$\widetilde{u}\in\mathcal{M}_r$ given by 
\[
T_{\widetilde{u}}\mathcal{M}_r=\ker \mathcal{H}_{q_1q_2}^\prime(\widetilde{u}).
\]
We denote by $\mathcal{J}_{\mathcal{M}_r}$ the restriction of $\mathcal{J}$ to $\mathcal{M}_r$, 
and by $\mathcal{J}^\prime_{\mathcal{M}_r}(\widetilde{u})$ the differential of $\mathcal{J}$ 
at $\widetilde{u}\in \mathcal{M}_r$ with respect to $\mathcal{M}_r$, i.e., the restriction of 
$\mathcal{J}^\prime(\widetilde{u})$ to $T_{\widetilde{u}}\mathcal{M}_r$.

An explicit expression for $\mathcal{J}^\prime_{\mathcal{M}_r}$ is given by
\begin{equation}\label{pr:TJ}
\mathcal{J}_{\mathcal{M}_r}^\prime(\widetilde{u})
=\mathcal{J}^\prime(\widetilde{u})
-\lambda(\widetilde{u}) \mathcal{H}_{q_1q_2}^\prime(\widetilde{u}),~
\lambda(\widetilde{u})
=\frac{\langle \mathcal{J}^\prime(\widetilde{u}), \widetilde{u}\rangle}
{\langle \mathcal{H}_{q_1q_2}^\prime(\widetilde{u}), \widetilde{u}\rangle} ~\text{for all } \widetilde{u}\in\mathcal{M}_r,
\end{equation}
(see \cite[Remark~3]{BBM}). 

As a direct consequence of Lemma~\ref{le:in}, we obtain the following result.
\begin{lemma}\label{le:PS}
Let $r>0$. Under the assumptions of Theorem~\ref{teorema1}, the functional $\mathcal{J}$ satisfies the Palais--Smale condition on $\mathcal{M}_r$.
\end{lemma}
\begin{proof}
Let $\big(\widetilde u_n\big)_n\subset\mathcal M_r$
with $\big(\mathcal J(\widetilde u_n)\big)_n$ bounded and
$\mathcal J'_{\mathcal M_r}(\widetilde u_n)\to 0$ in $(T_{\widetilde u_n}\mathcal M_r)^*$.

Since, from Lemma~\ref{le:coerciv}, $ \mathcal J$ is coercive on $\mathcal M_r$, the sequence $\big(\widetilde u_n\big)_n$ is bounded in
$\mathcal W$. Therefore, up to a subsequence,  $\widetilde u_n=(u_{1n}, u_{2n})\rightharpoonup \widetilde{u}_*=(u_{1*},u_{2*})$ in $\mathcal W$. 

By $(H_{\alpha\beta})$, we have $\mathcal I_{\alpha\beta}\neq\emptyset$.
For each $j\in \mathcal I_{\alpha\beta}$, Remark~\ref{re:esobW} implies that
$u_{jn}\to u_{j*}$ in $L^{q_j}(\Omega_j)$ and in $L^{q_j}(\Gamma_j)$.
Passing to the limit in the constraint defining $\mathcal{M}_r$, we infer that
$\widetilde{u}_*\in\mathcal{M}_r$.

Fix $i\in \{1, 2\}.$ As $\big(\widetilde{u}_n\big)_n$ is bounded in $\mathcal{W}$, we derive that 
\begin{equation}\label{eq:bis3.3}
\big(\|\nabla u_{in}\|_{ip_i}\big)_n\quad \text{is bounded}.
\end{equation}
As $\big(\widetilde{u}_n\big)_n$ is bounded in $\mathcal W$,
Remark~\ref{re:esobW} and $(H_{\rho\gamma})$ imply that
\begin{equation}\label{eq:3.3}
\big(\||\rho_i|^{\tfrac1{\zeta_i}} u_{in}\|_{i\zeta_i}\big)_n,~
\big(\||\gamma_i|^{\tfrac1{\eta_i}} u_{in}\|_{\partial i\eta_i}\big)_n
\quad\text{are bounded}.
\end{equation}
On the other hand, since $\mathcal{J}'_{\mathcal{M}_r}(\widetilde u_n)\to 0$
in $(T_{\widetilde u_n}\mathcal M_r)^*$, it follows from \eqref{pr:TJ}
that there exists a sequence $(\lambda_n)_n\subset\mathbb{R}$ such that
\begin{equation}\label{eq:3.diff}
\mathcal{J}'(\widetilde{u}_n)-\lambda_n \mathcal{H}_{q_1q_2}'(\widetilde{u}_n)
=K_{p_1p_2}'(\widetilde{u}_n)+k_{\zeta\eta}'(\widetilde{u}_n)
-\lambda_n \mathcal{H}_{q_1q_2}'(\widetilde{u}_n)\to 0\quad\text{in }\mathcal{W}^*.
\end{equation}
We claim that the sequence $\big(\lambda_n\big)_n$ is bounded. 
Indeed, evaluating \eqref{eq:3.diff} at $\widetilde u_n$, we obtain
\begin{equation}\label{eq:3.4}
\begin{aligned}
\sum_{i=1}^2\Big(
&\int_{\Omega_i} G_{iy}\big(x, F(\nabla u_{in})\big) F(\nabla u_{in})\,dx\\
&+\int_{\Omega_i}\rho_i |u_{in}|^{\zeta_i}\,dx
+\int_{\Gamma_i}\gamma_i |u_{in}|^{\eta_i}\,d\sigma\Big)-\lambda_n\big\langle \mathcal{H}_{q_1q_2}'(\widetilde{u}_n),\widetilde{u}_n\big\rangle
\to 0.
\end{aligned}
\end{equation}
Next, from  $(H_{G3})$ and \eqref{eq:equiv}, we have, a.e. on $\Omega_i$
\begin{equation}
F(\nabla u_{in}(x))G_{iy}(x, F (\nabla u_{in}(x)))\le C_{1i} \big(|\nabla u_{in}(x)| + |\nabla u_{in}(x)|^{p_i}\big).
\end{equation}
 Integrating over $\Omega_i$, and using
 the boundedness in \eqref{eq:bis3.3}, we infer that
\begin{equation}\label{eq:bis3.4}
\Big(\int_{\Omega_i}
G_{iy}\big(x, F(\nabla u_{in})\big)\,
F (\nabla u_{in})\,dx\Big)_n
\quad \text{is bounded}.
\end{equation}
Since $\big(\widetilde{u}_n\big)_n\subset\mathcal{M}_r$, we have
\begin{equation}\label{eq:bis3.5}
q_{\min} r \le
\langle \mathcal{H}_{q_1q_2}'(\widetilde{u}_n),\widetilde{u}_n\rangle
\le q_{\max} r\quad\text{for all }n\in\mathbb N.
\end{equation}
The convergence \eqref{eq:3.4}, together with the boundedness of the terms in \eqref{eq:3.3},  \eqref{eq:bis3.4}, and \eqref{eq:bis3.5} implies that
the sequence $(\lambda_n)_n$ is bounded. 

A similar argument to that used in the proof of \cite[Lemma~6]{BBM},
based on H\"older's inequality, the boundedness of $(\widetilde u_n)_n$
in $\mathcal W$, and Remark~\ref{re:esobW}, yields
\begin{equation}\label{eq:3.5}
\langle k'_{\zeta\eta}(\widetilde u_n),\widetilde u_n-\widetilde u_* \rangle \to 0,
\qquad
\lambda_n\langle j'_{q_1q_2}(\widetilde u_n),\widetilde u_n-\widetilde u_* \rangle \to 0,
\end{equation}
where in the second convergence we also use the boundedness of $(\lambda_n)_n$.

Evaluating \eqref{eq:3.diff} at  $\widetilde{u}_n-\widetilde{u}_*$ and using \eqref{eq:3.5}, we infer that
$
\langle K'_{p_1p_2}(\widetilde{u}_n),\, \widetilde{u}_n-\widetilde{u}_*\rangle \to 0.
$
Therefore, by Lemma~\ref{le:in}, we conclude that
$\widetilde{u}_n\to \widetilde{u}_*$ strongly in $\mathcal{W}$.
This completes the proof.

\end{proof}

\begin{remark}\label{re:finitgenus}
By Lemma~\ref{le:coerciv} and Lemma~\ref{le:PS}, the functional $\mathcal J$ is even
and bounded from below on $\mathcal M_r$, and satisfies the Palais--Smale condition
with respect to $\mathcal M_r$. Hence, standard results in genus theory imply
that every sublevel set of $\mathcal J$,
$
\mathcal J_b:=\{\widetilde u\in\mathcal M_r;\ \mathcal J(\widetilde u)\le b\},
$
has finite Krasnosel'skii genus (see, e.g., \cite[Lemma~9]{BF}).
\end{remark}

As a consequence of Lemmas~\ref{le:ig}, \ref{le:coerciv}, \ref{le:PS},
and Theorem~\ref{te:Sz}, the functional $\mathcal J$ admits infinitely many
distinct pairs of critical points $\pm\widetilde u_n$, $n\ge1$, on $\mathcal M_r$ (hence $\widetilde u_n\neq0$).
Each such critical point corresponds to a Lagrange multiplier $\lambda_n$, and
thus to an eigenpair $(\lambda_n, \pm\widetilde u_n)$ of problem~\eqref{eq:1.0}--\eqref{eq:1.00}.

To complete the proof of Theorem~\ref{teorema1}, it remains to verify that
$\lambda_n\to\infty$ as $n\to\infty$.
Fix $n\ge1$ and define
$
\mathcal J_n:=\{\widetilde u\in\mathcal M_r;~ \mathcal J(\widetilde u)\le n\}.
$
By Remark~\ref{re:finitgenus}, there exists an integer $j_n$ such that
$\gamma(\mathcal J_n)=j_n$.
Moreover, Lemma~\ref{le:ig} ensures the existence of a compact set
$K_n\subset \mathcal M_r\cap\mathcal E$ with $\gamma(K_n)=j_n+1$.
In particular, the set $\Lambda_{j_n+1}$ defined in Theorem~\ref{te:Sz} is nonempty.

By the monotonicity of the genus (see, e.g., \cite[Lemma~1.1]{Ra74}) and the definition of $j_n$, we obtain
$
\sup_A \mathcal J > n ~ \text{for any } A\in\Lambda_{j_n+1},
$
and consequently $c_n\ge n$.

\noindent
Since $\mathcal J$ is bounded from below, we have $c_1>-\infty$, and therefore
$-\infty<c_1\le \cdots \le c_n<\infty$.
Furthermore, by Lemma~\ref{le:PS}, the functional $\mathcal J$ satisfies the
Palais--Smale condition on $\mathcal M_r$, which implies that each $c_n$ is a
critical value of $\mathcal J|_{\mathcal M_r}$ (see \cite[Chapter~II, Section 5]{St}).

Consequently, for every $n\ge 1$ there exist $\lambda_n\in\mathbb R$ and
$\widetilde u_n=(u_{1n},u_{2n})\in\mathcal M_r$ such that
\begin{equation}\label{eq:3.kk}
\mathcal J'(\widetilde u_n)
=\lambda_n  \mathcal{H}_{q_1q_2}'(\widetilde u_n),\quad \mathcal J(\widetilde u_n)=c_n\ge n.
\end{equation}
In particular, \eqref{eq:3.kk} implies that
\[
\lambda_n \geq\frac{\langle \mathcal{J}^\prime (\widetilde{u}_n), \widetilde{u}_n\rangle}{q_{\max} r}\quad\text{for all } n\geq 1,
\]
\begin{equation}\label{eq:3.10}
\mathcal{J} (\widetilde{u}_n)\to \infty \quad\mbox{as } n\rightarrow \infty.
\end{equation}
In order to conclude the proof of Theorem~\ref{teorema1}, we only need to verify that \eqref{eq:3.10} implies
\begin{equation}\label{eq:3.f}
\begin{aligned}
\langle \mathcal J'(\widetilde u_n),\widetilde u_n\rangle
&=\sum_{i=1}^2\Big(
\int_{\Omega_i} G_{iy}\bigl(x,F(\nabla u_{in})\bigr)F(\nabla u_{in})\,dx\\
&\qquad
+\int_{\Omega_i}\rho_i |u_{in}|^{\zeta_i}\,dx
+\int_{\Gamma_i}\gamma_i |u_{in}|^{\eta_i}\,d\sigma
\Big)
\to\infty \quad\text{as } n\to\infty.
\end{aligned}
\end{equation}
As in the proof of Lemma~\ref{le:coerciv}, the inclusion $\big(\widetilde u_n\big)_n\subset \mathcal M_r$, together with
Remark~\ref{re:esobW} and the norm equivalence in Lemma~\ref{le:normatransmission}, implies that, for all $n\ge 1$,
\begin{equation}\label{eq:0potsmall}
\begin{aligned}
\mathcal S_{1n}
&:= \sum_{\zeta_i<p_{\min}}
\frac{1}{\zeta_i}\Big|\int_{\Omega_i}\rho_i|u_{in}|^{\zeta_i}\,dx\Big|
+ \sum_{\eta_i<p_{\min}}
\frac{1}{\eta_i}\Big|\int_{\Gamma_i}\gamma_i|u_{in}|^{\eta_i}\,d\sigma\Big| \\
&\quad\le C_1\Big(1
+\sum_{\zeta_i<p_{\min}}\|\nabla u_{in}\|_{ip_i}^{\zeta_i}
+\sum_{\eta_i<p_{\min}}\|\nabla u_{in}\|_{ip_i}^{\eta_i}\Big).
\end{aligned}
\end{equation}
Fix $i\in\{1,2\}.$ Integrating over $\Omega_i$ in \eqref{eq:HG2nou}  and using \eqref{eq:equiv}, we obtain
\[
\|\nabla u_{in}\|_{ip_i}^{p_i} \le C_{2i}\Big(1+ \int_{\Omega_i} G_i\big(x, F (\nabla u_{in})\big)\,dx\Big),
\]
which implies, for any $\theta_i>0$,
\[
\|\nabla u_{in}\|_{ip_i}^{\theta_i} \le C_{3i}\biggl(1+\Big(\int_{\Omega_i} G_i\big(x,F (\nabla u_{in})\big)\,dx\Big)^{\frac{\theta_i}{p_i}}\biggr)\quad \text{for all }n\ge 1.
\]
Combined with  \eqref{eq:0potsmall}, this estimate  leads to
\begin{equation}\label{eq:potsmall}
\mathcal S_{1n}\le C_{4} \biggl(1+\sum_{i=1}^{2}\Big(\int_{\Omega_i} G_i\bigl(x, F(\nabla u_{in})\bigr)\,dx\Big)^{\frac{\omega_i}{p_i}}\biggr)\quad \text{for all }n\ge 1,
\end{equation}
where $\omega_i:=\max\{\zeta_i,\eta_i\}$ if $\max\{\zeta_i,\eta_i\}<p_{\min}$,
and $\omega_i:=0$ otherwise.

On the other hand, from  \eqref{eq:potsmall},  and the definition of $\mathcal J$, we obtain
\begin{equation}\label{eq:infty}
\begin{aligned}
\mathcal J(\widetilde u_n)
&\le \sum_{i=1}^2\int_{\Omega_i} G_i\bigl(x,F(\nabla u_{in})\bigr)\,dx
+ \mathcal S_{2n}\\
&\quad + C_4\biggl(1
+\sum_{i=1}^{2}\Big(\int_{\Omega_i}
G_i\big(x,F(\nabla u_{in})\big)\,dx\Big)^{\frac{\omega_i}{p_i}}\biggr).
\end{aligned}
\end{equation}
where 
\[
\mathcal S_{2n}:=\sum_{\zeta_i\ge p_{\min}}
\frac{1}{\zeta_i}\int_{\Omega_i}\rho_i|u_{in}|^{\zeta_i}\,dx
+\sum_{\eta_i\ge p_{\min}}\frac{1}{\eta_i} \int_{\Gamma_i}\gamma_i|u_{in}|^{\eta_i}\,d\sigma.
\] 
Since $\mathcal S_{2n}\ge 0$ and $\mathcal J(\widetilde{u}_n)\to \infty$,
by $(H_{\rho\gamma})$ and \eqref{eq:3.10}, respectively,
it follows that the right-hand side of \eqref{eq:infty} tends to $\infty$.
As $\omega_i<p_{\min}$, this yields
\begin{equation}\label{eq:Gdiv}
\sum_{i=1}^2\int_{\Omega_i} G_i\big(x, F(\nabla u_{in})\big)\,dx + \mathcal S_{2n} \to \infty.
\end{equation}
Finally, integrating inequality \eqref{eq:yGy} over
$\Omega_i$, 
with $y=F (\nabla u_{in})$,
we obtain
\[
\sum_{i=1}^2\int_{\Omega_i}
G_{iy}\big (x, F (\nabla u_{in})\big ) F (\nabla u_{in})\,dx
\ge
\sum_{i=1}^2  \int_{\Omega_i}G_i\big(x, F (\nabla u_{in}\big)\big) \,dx.
\]
Together with \eqref{eq:potsmall}, this implies
\[
\begin{split}
\langle \mathcal J'(\widetilde u_n),\widetilde u_n\rangle&\ge \mathcal S_{2n}
+\sum_{i=1}^2 \int_{\Omega_i}G_i\big(x, F (\nabla u_{in})\big)\,dx\\
&\quad-
C_{5}\biggl(1+\sum_{i=1}^2\Big(\int_{\Omega_i}G_i\big(x,F (\nabla u_{in})\big)\,dx\Big)^{\frac{\omega_i}{p_i}}\biggr).
\end{split}
\]
Consequently, from \eqref{eq:Gdiv} and $\omega_i<p_{\min}$, we obtain $\langle \mathcal J'(\widetilde u_n),\widetilde u_n\rangle \to \infty.$
This completes the proof of Theorem~\ref{teorema1}.

\begin{remark}\label{re:normanondeg}

Assume that $\mathcal I_{\alpha\beta}=\{1,2\}$.
Then 
\begin{equation}\label{eq:p.norme}
\|\widetilde u\|_{\alpha\beta}
:=
\sum_{i=1}^2\bigl(\|\nabla u_i\|_{ip_i}+s_{i\alpha\beta}(u_i)\bigr),
\quad \widetilde u=(u_1,u_2)\in\mathcal W,
\end{equation}
defines a norm on $\mathcal W$ which is equivalent to the standard norm
introduced in~\eqref{eq:1.nw}.
Here $s_{i\alpha\beta}$ is defined as in~\eqref{eq:seminorme} with $j$ replaced by $i$.
Moreover, for all $\widetilde u\in\mathcal M_r$,
\begin{equation}\label{eq:talbe}
\|\widetilde u\|_{\alpha\beta}\le\sum_{i=1}^2\|\nabla u_i\|_{ip_i}+\sum_{i=1}^2 (q_i r)^{1/q_i}.
\end{equation}
Similarly
$\|u\|_{i\alpha\beta}:=\|\nabla u\|_{ip_i}+s_{i\alpha\beta}(u),
~ u\in W^{1,p_i}(\Omega_i),
$
defines a norm on $W^{1,p_i}(\Omega_i)$ equivalent to the standard one, and
$
\|u_i\|_{i\alpha\beta}
\le
\|\nabla u_i\|_{ip_i}
+(q_i r)^{1/q_i}$
for all $\widetilde u=(u_1,u_2)\in\mathcal M_r.$
Working with these equivalent norms, the conclusion of
Theorem~\ref{teorema1} remains valid if assumption
$(H_{\rho\gamma})$ is replaced by the weaker condition
\begin{itemize}
\item[$(H'_{\rho\gamma})$]
$\rho_i \in L^{\infty}(\Omega_i)$ and $\gamma_i \in L^{\infty}(\Gamma_i)$ for $i=1,2$.
Moreover, if $\zeta_i \ge p_i$, we assume $\rho_i \ge 0$ a.e.\ in $\Omega_i$,
and if $\eta_i \ge p_i$, we assume $\gamma_i \ge 0$ a.e.\ on $\Gamma_i$.
\end{itemize}
\end{remark}

\section{ Proof of Theorem~\ref{teorema2}}\label{sec:teorema2}

Under the assumptions of Theorem~\ref{teorema2}, choosing $\widetilde v=\widetilde{u}_\lambda$ in \eqref{eq:1.def} implies that any eigenvalue $\lambda$ of problem~\eqref{eq:1.0}--\eqref{eq:1.00} is nonnegative.
 
Clearly, $\lambda=0$ is not an eigenvalue under assumption $(h_{\rho\gamma})$.
Hence, all eigenvalues lie in $(0,\infty)$.

For $\lambda > 0$, define the functional
\begin{equation}\label{Jl}
\mathcal{J}_\lambda:\mathcal{W} \to \mathbb{R},~\mathcal{J}_\lambda(\widetilde{u}):= K_{p_1p_2}(\widetilde{u}) + k_{p_1p_2}(\widetilde{u})-\lambda \mathcal{H}_{q_1q_2}(\widetilde{u}).
\end{equation}
According to  \eqref{eq:1.def}, $\lambda>0$ is an eigenvalue of problem \eqref{eq:1.0}--\eqref{eq:1.00} if and only if there exists a critical point $\widetilde u_\lambda\in \mathcal W\setminus\{0\}$ of $\mathcal{J}_\lambda$, i.e. $\mathcal{J}_\lambda^\prime(\widetilde u_\lambda)=0.$

\smallskip
\textbf{Case $(a)\quad$} The proof relies on the Palais--Smale condition for $\mathcal J_\lambda$ on $\mathcal W$.
\begin{lemma}
Let $\lambda>0.$ Under the assumptions of Theorem~\ref{teorema2}, Case~$(a)$, 
the functional $\mathcal{J}_\lambda$ satisfies the Palais--Smale condition on
$\mathcal W$.
\end{lemma}
\begin{proof}
Let $(\widetilde u_n)\subset\mathcal W$ be such that 
$
\big(\mathcal J_\lambda(\widetilde u_n)\big)_n$  is bounded in $\mathcal W$
and
$\mathcal J_\lambda'(\widetilde u_n)\to 0$ in $\mathcal W^*.$

We claim that $\big(\widetilde u_n\big)$ is bounded in $\mathcal W$. 

Let
$\theta\in(p_{\max},\,q_{\min})$.
Fix $i\in \{1, 2\}$ and, for each $n\in \mathbb N$, define
\[
\mathcal{K}_{in\theta}(x):=\theta G_i(x, F (\nabla u_{in}(x)))- F (\nabla u_{in}(x))G_{iy}(x, F (\nabla u_{in}(x))),~x\in\Omega_i.
\]
By integrating over $\Omega_i$ and applying $(h_{G4})$, we obtain, for all $n\in \mathbb N$
\begin{equation}\label{eq:estimareA}
\begin{split}
\int_{\Omega_i}\mathcal{K}_{in\theta}\,dx\ge
(\theta-p_i)\int_{\Omega_i}G_i(x, F(\nabla u_{in})dx
-
c_i\int_{\Omega_i} F(\nabla u_{in})^{\delta_i}dx.
\end{split}
\end{equation}
Fix $\varepsilon_0>0$. By Young's inequality, we have
\[
c_i y^{\delta_i}\le \varepsilon_0 y^{p_i} + C_{i\varepsilon_0}\quad \text{for all }y\ge0,
\]
 for some $C_{i\varepsilon_0}>0$.
Taking $y = F(\nabla u_{in}(x))$ and integrating over $\Omega_i$, we derive
\[
-c_i\int_{\Omega_i}F (\nabla u_{in})^{\delta_i}\,dx
\ge -\varepsilon_0\int_{\Omega_i}F (\nabla u_{in})^{p_i}\,dx - c_i C_{i\varepsilon_0}\operatorname{meas}(\Omega_i).
\]
Together with \eqref{eq:estimareA}, the last estimate leads to
\begin{equation}\label{eq:0A}
\begin{split}
\int_{\Omega_i}\mathcal{K}_{in\theta}\,dx &\ge (\theta-p_i)\int_{\Omega_i}G_i(x, F(\nabla u_{in})\,dx
-\varepsilon_0\int_{\Omega_i}F (\nabla u_{in})^{p_i}\,dx - C_{1i}.
 \end{split}
\end{equation}
From \eqref{eq:0A}, using \eqref{eq:HG2nou} and \eqref{eq:equiv}, we obtain
\[
\int_{\Omega_i}\mathcal{K}_{in\theta}\,dx
\ge \big(d_i(\theta-p_i)-\varepsilon_0\big)
\int_{\Omega_i}F(\nabla u_{in})^{p_i}\,dx - C_{1i}.
\]
Choosing $\varepsilon_0>0$ such that
$\varepsilon_0<\min\{d_1(\theta-p_1),d_2(\theta-p_2)\}$,
we derive
\[
\int_{\Omega_i}\mathcal{K}_{in\theta}\,dx
\ge C_{2i}\|\nabla u_{in}\|_{ip_i}^{p_i}-C_{1i},
\]
with $C_{2i}=m\big(d_i(\theta-p_i)-\varepsilon_0\big)>0$. Hence,
\begin{equation}\label{eq:mp1}
\theta K_{p_1p_2}(\widetilde u_n)-\langle K_{p_1p_2}'(\widetilde u_n),\widetilde u_n\rangle=\sum_{i=1}^{2}\int_{\Omega_i}\mathcal{K}_{in\theta}\,dx \ge
C_3\sum_{i=1}^{2}\|\nabla u_{in}\|_{ip_i}^{p_i}
 - C_{4}.
\end{equation}
Let $l\in\{1,2\}$ be as in $(h_{\rho\gamma})$. For $\widetilde u\in\mathcal W$, set
\[
s_{l\rho\gamma}(u_l):=\left(\int_{\Omega_l}\rho_l | u_l|^{p_l}\,dx+\int_{\Gamma_l}\gamma_l |u_l|^{p_l}\,d\sigma\right)^{\frac{1}{p_l}}.\]
By Lemma~\ref{le:normatransmission}
(with $j$ replaced by $l$, $\alpha_{j}$ by $\rho_l$,
and $\beta_{j}$ by $\gamma_l$), 
\begin{equation}\label{eq:defnorma}
\|\widetilde u\|_{l\rho\gamma}
:=\sum_{i=1}^2\| \nabla u_i \|_{ip_i}+s_{l\rho\gamma}(u_l),~\widetilde u\in \mathcal W,
\end{equation}
defines a norm on $\mathcal W$ equivalent to the norm defined in \eqref{eq:1.nw}.

\noindent
On the other hand, since $\theta\in(p_{\max},q_{\min})$, it follows from $(h_{\rho\gamma})$ and $(H_{\alpha\beta})$ that
\begin{equation}\label{eq:mp1kj}
\begin{aligned}
\theta k_{p_1p_2}(\widetilde u_n)
-\langle k_{p_1p_2}'(\widetilde u_n),\widetilde u_n\rangle
&=\sum_{i=1}^2\Big(\frac{\theta}{p_i}-1\Big)
\Big(\|\rho_i^{1/p_i}u_{in}\|_{ip_i}^{p_i}
+\|\gamma_i^{1/p_i}u_{in}\|_{\partial i\,p_i}^{p_i}\Big)\\
&\ge \Big(\frac{\theta}{p_l}-1\Big)s_{l\rho\gamma}(u_{ln})^{p_l}\ge0,\\[4pt]
\theta \mathcal H_{q_1q_2}(\widetilde u_n)
-\langle \mathcal H'_{q_1q_2}(\widetilde u_n),\widetilde u_n\rangle
&=\sum_{i\in\mathcal I_{\alpha\beta}}
\Big(\frac{\theta}{q_i}-1\Big)\\
&\qquad\cdot
\Big(\|\alpha_i^{1/q_i}u_{in}\|_{iq_i}^{q_i}
+\|\beta_i^{1/q_i}u_{in}\|_{\partial iq_i}^{q_i}\Big)<0 .
\end{aligned}
\end{equation}
Since $\big(\mathcal J_\lambda(\widetilde u_n)\big)_n$ is bounded in $\mathcal W$  and $\mathcal J_\lambda'(\widetilde u_n)\to 0$ in $\mathcal W^*$,
there exists $N_0\ge 0$ such that 
\begin{equation}\label{eq:mp002}
\parallel\mathcal J_\lambda'(\widetilde u_n)\parallel_{\mathcal W^*}\le 1~\text{and}~\mid\mathcal J_\lambda(\widetilde u_n)\mid\le C_5~\text{for all }n\ge N_0.
\end{equation}
Therefore, 
\begin{equation}\label{eq:mp02}
\big|\langle \mathcal J_\lambda'(\widetilde u_n),\widetilde u_n\rangle\big|
\le \|\mathcal J_\lambda'(\widetilde u_n)\|_{\mathcal W^*} \|\widetilde u_n\|
\le C_6\|\widetilde u_n\|_{l\rho\gamma}~\text{for all }n\ge N_0.
\end{equation}

From \eqref{eq:mp1}, \eqref{eq:mp1kj}, \eqref{eq:mp002}, and \eqref{eq:mp02},  for all $n\ge N_0$, it follows that
\begin{equation}\label{eq:mp3}
\begin{split}
 \theta C_5+ C_6\|\widetilde u_n\|_{l\rho\gamma}
\ge \theta\mathcal J_\lambda(\widetilde u_n)-\langle \mathcal J_\lambda'(\widetilde u_n),\widetilde u_n\rangle
\ge C_7\Big(\sum_{i=1}^{2}\|\nabla u_{in}\|_{ip_i}^{p_i} + s_{l\rho\gamma}(u_{ln})^{p_l}\Big),
\end{split}
\end{equation}
where $C_7=\min\{C_3, \theta/p_l-1\}>0$.
 Denote 
\begin{equation}\label{eq:defTM}
\begin{aligned}
T(\widetilde u)
&:=\sum_{i=1}^{2}\|\nabla u_i\|_{ip_i}^{p_i}
   + s_{l\rho\gamma}(u_l)^{p_l},\\
M(\widetilde u)
&:=\max\{\|\nabla u_1\|_{1p_1},\,\|\nabla u_2\|_{2p_2},\,
        s_{l\rho\gamma}(u_l)\},\qquad
\widetilde u=(u_1,u_2)\in\mathcal W .
\end{aligned}
\end{equation}
Assume by contradiction that
$
\|\widetilde u_n\|_{l\rho\gamma}\to\infty
~ \text{as } n\to\infty$.
In particular, $T(\widetilde u_n)\to\infty$, and consequently
$M(\widetilde u_n)\to\infty$.
Hence, there exists $N_1\ge N_0$ such that
\begin{equation}\label{eq:0estimareMn}
M(\widetilde u_n)\ge \frac{1}{3}\|\widetilde u_n\|_{l\rho\gamma}\ge 1
~ \text{for all } n\ge N_1.
\end{equation}
Moreover,
\begin{equation}\label{eq:0estimareTn}
T(\widetilde u_n) \ge M(\widetilde u_n)^{p_{\min}}
~ \text{for all } n\ge N_1.
\end{equation}
From \eqref{eq:mp3}, \eqref{eq:0estimareMn}, and \eqref{eq:0estimareTn} we further infer that
\begin{equation}
\begin{split}
\theta C_5&+ C_6\|\widetilde u_n\|_{l\rho\gamma}
\ge  C_7 T(\widetilde{u}_n)\ge C_7 M(\widetilde{u}_n)^{p_{\min}}\ge \frac{C_7}{3^{p_{\min}}}\|\widetilde u_n\|_{l\rho\gamma}^{p_{\min}}
\end{split}
\end{equation}
for all $n\ge N_1.$ Since $p_{\min}>1$, the right-hand side grows faster than $\|\widetilde u_n\|_{l\rho\gamma}$ as $n\to\infty$,
which contradicts $\|\widetilde u_n\|_{l\rho\gamma}\to\infty$
 as $n\to\infty.$ Therefore, $\big(\widetilde u_n\big)_n$ is bounded in $\mathcal W$.

Consequently, up to a subsequence, $\widetilde u_n \rightharpoonup \widetilde u$ in $\mathcal W$.
Using a standard argument (see, e.g. \cite[Lemma 6]{BBM}), we have
\[
\langle k_{p_1p_2}',  \widetilde u_n-\widetilde u\rangle \to 0,\quad \lambda \langle \mathcal{H}_{q_1q_2}'(\widetilde u_n), \widetilde u_n-\widetilde u\rangle \to 0\quad\text{as }n\to\infty.
\]
This combined with
\[
\langle\mathcal J_\lambda'(\widetilde u_n), \widetilde u_n-\widetilde u\rangle
=
\langle K_{p_1p_2}'(\widetilde u_n), \widetilde u_n-\widetilde u\rangle
+
\langle (k_{p_1p_2}'-\lambda \mathcal{H}_{q_1q_2}')(\widetilde u_n), \widetilde u_n-\widetilde u\rangle
\to 0 ~ \text{in } \mathcal W^*,
\]
yields $\langle K_{p_1p_2}'(\widetilde u_n), \widetilde u_n-\widetilde u\rangle \to 0.
$
Therefore, by Lemma~\ref{le:in}, we conclude that
$
\widetilde u_n \to \widetilde u ~\text{in } \mathcal W.
$
Hence, $\mathcal J_\lambda$ satisfies the Palais--Smale condition on $\mathcal W$. 

\end{proof}

Next, by Remark~\ref{re:esobW} and $(H_{pq})$, it follows that
\begin{equation}\label{eq:mp4}
\mathcal{H}_{q_1q_2}(\widetilde u)\le
C_8 \sum_{i\in \mathcal I_{\alpha\beta}} \|\widetilde u\|_{l\rho\gamma}^{q_i}\quad\text{for all }\widetilde u\in \mathcal W.
\end{equation}
On the other hand, for $\|\widetilde u\|_{l\rho\gamma}<1$, we have
\begin{equation}\label{eq:0mp4}
T(\widetilde{u})\ge \sum_{i=1}^2 \|  \nabla u_i \|_{ip_i}^{p_{\max}} +s_{l\rho\gamma}(u_{l})^{p_{\max}}\ge 3^{1-p_{\max}}\|\widetilde u\|_{l\rho\gamma}^{p_{\max}}.
\end{equation}
By $(h_{\rho\gamma})$, \eqref{eq:mgK} with $C_2=0$ (since $(H_{G2})$ holds with $y_{0i}=0$), and \eqref{eq:mp4}-\eqref{eq:0mp4}, we obtain
\begin{equation}\label{eq:mp5}
\mathcal J_\lambda(\widetilde u)
\ge 
C_9 T(\widetilde u)-\lambda \mathcal{H}_{q_1q_2}(\widetilde u)
\ge
C_{10} \|\widetilde u\|_{l\rho\gamma}^{p_{\max}}
-
2\lambda C_8
\|\widetilde u\|_{l\rho\gamma}^{q_{\min}}
\quad
\text{for }~ \|\widetilde u\|_{l\rho\gamma}<1.
\end{equation}
Since $p_{\max}<q_{\min}$, it follows from \eqref{eq:mp5} that there exists $r_\lambda\in(0,1)$ satisfying
\[
m_\lambda
:=
\inf_{\|\widetilde u\|_{l\rho\gamma}=r_\lambda}
{\mathcal J}_\lambda(\widetilde u)
>
{\mathcal J}_\lambda((0, 0))
=0.
\]
Fix $t>1$ and consider the constant function $(t,t)\in\mathcal W$. By $(H_{\alpha\beta})$ we have $\mathcal{H}_{q_1q_2}\big((t, t)\big)>0$. As $q_{\min}>p_{\max}$, we obtain that
\[ 
\mathcal J_\lambda\big((t, t)\big)
=
k_{p_1p_2}\big((t, t)\big)-\lambda \mathcal{H}_{q_1q_2}\big((t, t)\big)
\le
C_{11} t^{p_{\max}}-\lambda C_{12}  t^{q_{\min}}.
\]
 In particular, $\mathcal J_\lambda\big((t, t)\big)\to -\infty
~\text{as }t\to\infty$. Hence, there exists $t_0>1$ such that
$
\mathcal J_\lambda\big((t_0, t_0)\big)<0$ and $\|(t_0, t_0)\|_{l\rho\gamma}>r_\lambda$,
where the latter follows from the assumption $(h_{\rho\gamma})$.

By the Mountain Pass Theorem (see, e.g., \cite[Theorem~2.2]{Ra}),
there exists a critical point $\widetilde u_\lambda\in\mathcal W$
of $\mathcal J_\lambda$ such that
$
\mathcal J_\lambda(\widetilde u_\lambda)\ge r_\lambda>0.
$
Hence, $\widetilde u_\lambda\neq 0$.
This completes the proof of Theorem~\ref{teorema2}, Case~$(a)$.

\smallskip
\textbf{Case $(b)\quad$}
The proof is based on the coercivity of the functional
$\mathcal J_\lambda$ on $\mathcal W$.

\begin{lemma}\label{le:coerb}
Let $\lambda>0$. Under the assumptions of Theorem~\ref{teorema2}, Case~$(b)$,
the functional $\mathcal J_\lambda$
is coercive on $\mathcal W$.
\end{lemma}

\begin{proof}
As in Case~$(a)$, we work with the norm $\|\cdot\|_{l\rho\gamma}$ on $\mathcal W$
defined in \eqref{eq:defnorma}, where $l\in\{1,2\}$ is fixed
according to assumption $(h_{\rho\gamma})$.

Assume that $\|\widetilde u\|_{l\rho\gamma}\to\infty$.
This implies that $M(\widetilde u)\to\infty$ and,
for $\|\widetilde u\|_{l\rho\gamma}\ge 3$, we have
\[
T(\widetilde u)\ge M(\widetilde u)^{p_{\min}}
\ge \frac{1}{3^{p_{\min}}}\|\widetilde u\|_{l\rho\gamma}^{p_{\min}},
\]
where $M$ and $T$ are defined in \eqref{eq:defTM}, Case~(a).
Combined with \eqref{eq:mgK} and \eqref{eq:mp4}, this yields
\begin{equation*}
\begin{split}
\mathcal J_\lambda(\widetilde u)
&=K_{p_1p_2}(\widetilde u)+k_{p_1p_2}(\widetilde u)-\lambda \mathcal{H}_{q_1q_2}(\widetilde u)
\ge T(\widetilde u)-2\lambda C_8\|\widetilde u\|_{l\rho\gamma}^{q_{\max}}\\
&\ge
\frac{1}{3^{p_{\min}}}\|\widetilde u\|_{l\rho\gamma}^{p_{\min}}
-2\lambda C_8\|\widetilde u\|_{l\rho\gamma}^{q_{\max}}\quad \text{for }\|\widetilde u\|_{l\rho\gamma}\ge 3.
\end{split}
\end{equation*}
Since $p_{\min}>q_{\max}$, it follows that $\mathcal J_\lambda(\widetilde u)\to \infty$ as $\|\widetilde u\|_{l\rho\gamma}\to\infty.$
Therefore, $\mathcal J_\lambda$ is coercive on $\mathcal W$. 
\end{proof}
Next, we claim that for every $\lambda>0$, problem~\eqref{eq:1.0}--\eqref{eq:1.00}
admits a nontrivial weak solution.

Fix an arbitrary $\lambda>0.$
Since $\mathcal{J}_\lambda$ is coercive and weakly lower semicontinuous on $\mathcal W$
(see Lemma~\ref{le:coerb}),
the direct method of the calculus of variations yields the existence of a global minimizer
$\widetilde u_*\in\mathcal W$ such that
$
\mathcal J_\lambda(\widetilde u_*)
=\min\bigl\{\mathcal J_\lambda(\widetilde u);\ \widetilde u\in\mathcal W\bigr\}.
$

To verify that $\widetilde u_*\not\equiv 0$, it suffices to prove that there exists
$\widetilde u\in\mathcal W$ such that $\mathcal J_\lambda(\widetilde u)<0.$

Take the constant function $(t, t)\in \mathcal W$ with $t>0$.
Then $K_{p_1p_2}\big((t, t)\big)=0$ and, from assumptions $(H_{\alpha\beta})$, $(h_{\rho\gamma})$, we obtain
\[
\mathcal J_\lambda\big((t, t)\big)=k_{p_1p_2}\big((t, t)\big)-\lambda \mathcal{H}_{q_1q_2}\big((t, t)\big)
\le C_{14} t^{p_{\min}}-\lambda C_{15}t^{q_{\max}}<0
\]
for $t>0$ small enough, as $p_{\min}>q_{\max}$ and $\mathcal I_{\alpha\beta}\neq\emptyset$.

This implies $\mathcal J_\lambda(\widetilde u_*)<0$, therefore
$\widetilde u_*\not\equiv 0$.
This completes the proof of Case~$(b)$.

\begin{remark}\label{re:zetaetaactive}
In Theorem~\ref{teorema2} we took $\zeta_i=\eta_i=p_i$ for simplicity.
The arguments of Section~\ref{sec:teorema2} extend to
$\zeta_i\in(1,p_i^*)$, $\eta_i\in(1,p_{i*})$ when the comparison with the
left-hand side exponents is made only for the active potential terms.
More precisely, set
\[
\mathcal I_{\rho}:=\Big\{ m\in\{1,2\};~ \int_{\Omega_m}\rho_m\,dx>0 \Big\},
\quad
\mathcal I_{\gamma}:=\Big\{ m\in\{1,2\};~ \int_{\Gamma_m}\gamma_m\,dx>0\Big\}.
\]
Under the assumption $(h_{\rho\gamma}),$ define
\[
(\zeta, \eta)_{\min}
:=
\min\{\zeta_j, \eta_k;\, j\in\mathcal I_{\rho},\, k\in\mathcal I_{\gamma}\},
\quad
(\zeta, \eta)_{\max}
:=
\max\{\zeta_j, \eta_k;\, j\in\mathcal I_{\rho},\, k\in\mathcal I_{\gamma}\}.
\]
With these conventions, the conclusions of Theorem~\ref{teorema2}
remain valid if $\zeta_i=\eta_i=p_i$ is replaced by
 $(\zeta,\eta)_{\max}<q_{\min}$ in Case~$(a)$ or $(\zeta,\eta)_{\min}>q_{\max}$ in Case~$(b)$,
while keeping all other assumptions unchanged.
\end{remark}
Corollary~\ref{cor:singledomain}  is an immediate consequence of
Theorems~\ref{teorema1} and~\ref{teorema2} together with
Remark~\ref{re:zetaetaactive}.





\end{document}